\documentclass[a4paper]{article}

\usepackage{amsmath}
\usepackage{amssymb}
\usepackage[pdftex,bookmarksnumbered=true]{hyperref}
\usepackage[pdftex]{graphicx,color} 

\def\modif{}
\usepackage{pgf,tikz}
\usepackage{mathrsfs}
\usetikzlibrary{arrows}
   {\begin{center}\begin{minipage}{.8\textwidth}%
      \scriptsize{\bf Abstract.}\enskip}%
   {\end{minipage}\end{center}}%
\newtheorem{lemma}[subsection]{Lemma}
\newtheorem{proposition}[subsection]{Proposition}
\newtheorem{theorem}[subsection]{Theorem}
\newtheorem{corollary}[subsection]{Corollary}
\newtheorem{claim}[subsection]{Claim}
\newtheorem{fact}[subsection]{Fact}
\newenvironment{remark}%
   {\refstepcounter{subsection}%
        \medbreak\noindent{\bf Remark \thesubsection\space}}%
   {\par\medbreak}%
\newenvironment{example}%
   {\refstepcounter{subsection}%
        \medbreak\noindent{\bf Example \thesubsection\space}}%
   {\par\medbreak}%
\newenvironment{proof}[1][\unskip]%
   {\medbreak\noindent{\it Proof #1:\space}}%
   {\par\noindent\vrule height 5pt width 5pt depth 0pt\smallbreak}%
\newcommand{\df}{\bf}

\let\sauvegardetiret=\-
\renewcommand{\-}[1]{\ifx#1-\penalty10000\hbox{-\relax}\penalty10000\else\sauvegardetiret#1\fi}
\newcommand{\tq}{\,\big/\ }
\newcommand{\NN}{{\rm\bf N}}
\newcommand{\ZZ}{{\rm\bf Z}}
\newcommand{\QQ}{{\rm\bf Q}}
\newcommand{\RR}{{\rm\bf R}}
\newcommand{\cA}{{\cal A}}
\newcommand{\cB}{{\cal B}}
\newcommand{\cC}{{\cal C}}
\newcommand{\cD}{{\cal D}}
\newcommand{\cE}{{\cal E}}

\newcommand{\cL}{{\cal L}}

\newcommand{\cQ}{{\cal Q}}

\newcommand{\cS}{{\cal S}}
\newcommand{\cT}{{\cal T}}
\newcommand{\cU}{{\cal U}}

\newcommand{\cW}{{\cal W}}

\newcommand{\cZ}{{\cal Z}}

\renewcommand{\tq}{\mathop{:}}

\newcommand{\conj}{\bigwedge}
\newcommand{\cconj}{\mathop{\hbox{$\conj\mskip-11mu\relax\conj$}}\limits}

\newcommand{\Supp}{\mathop{\rm Supp}}
\newcommand{\Card}{\mathop{\rm Card}}

\title{Polytopes and simplexes in $p$-adic fields}
\author{Luck Darni\`ere\footnote{D\'epartement de math\'ematiques, Facult\'e
des sciences, 2 bd Lavoisier, 49045 Angers, France.}}

\begin{document}

\maketitle

\begin{abstract}
  We introduce topological notions of polytopes and simplexes, the
  latter being expected to fulfill in $p$\--adically closed fields the
  function of real simplexes in the classical results of triangulation
  of semi-algebraic sets over real closed fields. We prove that the
  faces of every $p$\--adic polytope are polytopes and that they form
  a rooted tree with respect to specialisation. Simplexes are then
  defined as polytopes whose faces tree is a chain. Our main result is
  a construction allowing to divide every $p$\--adic polytope in a
  complex of $p$\--adic simplexes with prescribed faces and shapes. 
\end{abstract}

\section{Introduction}
\label{se:intro}

Throughout all this paper we fix a $p$\--adically closed field
$(K,v)$. The reader unfamiliar with this notion may restrict to the
special case where $K=\QQ_p$ or a finite extension of it, and $v$ is its
$p$\--adic valuation. We let $R$ denote the valuation ring of $v$, and
$\Gamma=v(K)$ its valuation group (augmented with one element $+\infty=v(0)$).
In this introductory section we present informally what we are aiming
at. Precise definitions will be given in Section~\ref{se:notation}
and at the beginning of Section~\ref{se:p-adic}.

Our long-term objective is to set a triangulation theorem which would
be an acceptable analogue over $K$ of the classical triangulation of
semi-algebraic sets over the reals. Polytopes and simplexes in $\RR^m$
are well known to have the following properties, among others (see for
example \cite{boch-cost-roy-1998} or \cite{drie-1998}). 
\begin{description}
  \item[(Sim)]
    They are bounded subsets of $\RR^m$ which can be described by a
    finite set of linear inequalities of a very simple form. 
  \item[(Fac)]
    There is a notion of ``faces'' attached to them with good
    properties: every face of a polytope $S$ is itself a polytope; if
    $S''$ is a face of $S'$ and $S'$ a face of $S$ then $S''$ is a
    face of $S$; the union of the proper faces of $S$ is a
    partition of its frontier.
\item[(Div)]
    Last but not least, every polytope can be divided in simplexes by
    a certain uniform process of ``Barycentric Division'' which offers
    a good control both on their shapes and their faces. 
\end{description}

The goal of the present paper is to build a $p$\--adic counterpart of
real polytopes and simplexes having similar properties. Obviously
there is no direct translation of concepts like linear {\em
inequalities} and {\em Barycentric} Division to {\em non-ordered}
fields, such as the $p$\--adic ones. Nevertheless we want our
$p$\--adic polytopes and simplexes to be defined by conditions which
are as simple as possible, to obtain a notion of faces satisfying all
the above properties, and most of all to develop a flexible and
powerful division tool.

This is achieved here by first introducing and studying certain
subsets of $\Gamma^m$ called ``largely continuous precells mod $N$'', for a
fixed $m$\--tuple $N$ of positive integers. These sets will
be defined by a very special triangular system of linear inequalities
and congruence relations mod $N$. In particular they are defined
simply by linear inequalities in the special case where $N=(1,\dots,1)$
(again, see Section~\ref{se:notation} for precise definitions and
basic examples). 

\medskip

This paper, which is essentially self-contained, is organised as
follows. The general properties of subsets of $\Gamma^m$ defined by
conjunctions of linear inequalities and congruence conditions are
studied in section~\ref{se:face-proj}. Property (Fac) is proved there
to hold true for largely continuous precells mod $N$ (while
property (Sim) is a by-product of their definition).
Section~\ref{se:bound-fun} is devoted to two technical properties
preparing the proof of our main result, a construction analogous to
(Div) in our context. We call this ``Monohedral Division'' (see
below). Section~\ref{se:division} is devoted to its proof.

We then return to the $p$\--adic context in the final
section~\ref{se:p-adic}. By taking inverse images of largely
continuous precells by the valuation $v$ (which maps $K^m$ onto $\Gamma^m$)
and restricting them to certain subsets of $R^m$, we transfer all the
definitions and results built in $\Gamma^m$ in the previous sections to
$K^m$, especially the Monohedral Division (which becomes in this
context the ``Monotopic Division'', Theorem~\ref{th:div-p-adic}). This
latter result paves the way towards a triangulation of semi-algebraic
$p$\--adic sets, to appear in a further paper. 

\paragraph{Monohedral division.}

In addition to (Sim) and (fac), every largely continuous precell $A$ mod
$N$ has one more remarkable property which real polytopes are
lacking: its proper faces, ordered by specialisation\footnote{The {\df
specialisation pre-order} of the subsets of a topological space is
defined as usually by $B\leq A$ if and only if $B$ is contained in the
closure or $A$.\label{fo:spec-ord}}\!, form a rooted tree
(Proposition~\ref{pr:face-egale-proj}(\ref{it:face-lattice}). When
this tree is a chain, we say that $A$ is ``monohedral''.

Among real polytopes of a given dimension, the simplexes are those
whose number of facets is minimal: a polytope $A\subseteq\RR^m$ of dimension $d$
has at least $d+1$ facets, and it is a simplex if and only if it has
exactly $d+1$ facets (see Corollary~9.5 and Corollary 12.8 in
\cite{bron-1983}). We expect largely continuous precells to fulfill in
$\Gamma^m$ a function similar to polytopes in $\RR^m$, and the monohedral
ones (whose ordered set of faces is in a sense the simplest possible
tree) to fulfill a function similar to simplexes.

Indeed our main result, the ``Monohedral Division''
(Theorem~\ref{th:mono-div}), provides in our context a powerful tool
very similar to (Div), the Barycentric Division of real polytopes. It
provides in particular a ``Monohedral Decomposition''
(Theorem~\ref{th:mono-dec}) which says that every largely continuous
precell mod $N$ in $\Gamma^m$ is the disjoint union of a complex of {\em
monohedral} largely continuous precells mod $N$. The latter result is in
analogy with the situation in the real case, where every polytope
can be divided in simplexes forming a simplicial complex. 

But the Barycentric Division in $\RR^m$ says much more than this. Roughly
speaking, given a polytope $A$ and a simplicial complex $\cD$
partitioning the frontier of $A$, it makes it possible to build a simplicial
complex $\cC$ by partitioning $A$ and ``lifting'' $\cD$, in the sense
that for every $C$ in $\cC$, the faces $D$ of $C$ which are outside
$A$ belong to $\cD$. Moreover, given a positive function $\varepsilon:B
\to\RR$ (where $B$ is any proper face of $A$), the shapes of the elements
of $\cC$ can be required to satisfy the following condition: for every
$D$ in $\cD$ there is a unique $C\in\cC$ such that $D$ is the largest
proper face of $C$ in $\cC$, and in that case the distance of any point
$x\in C$ to its projection $y$ onto $B$ is smaller than $\varepsilon(y)$ (see
Figure~\ref{fi:div}, where the dotted curve shows how the $\varepsilon$ function
controls the shapes of the elements of $\cC$ whose largest proper face
outside $A$ is contained in the lower facet $B$). 

\begin{figure}[h]
\begin{center}

  \begin{tikzpicture}[scale=0.5]
    \coordinate (A) at (1,0);
    \coordinate (B) at (3,0);
    \coordinate (C) at (5,0);
    \coordinate (D) at (6,0);
    \coordinate (E) at (7,1);
    \coordinate (F) at (6.5,2.5);
    \coordinate (G) at (6,4);
    \coordinate (H) at (3,4.5);
    \coordinate (I) at (0,3);
    \coordinate (J) at (0.5,1.5);
    \coordinate (AB) at (2,0.5);
    \coordinate (CD) at (4.5,1);
    \coordinate (GH) at (4,3);
    \coordinate (IJ) at (3,2);
    \coordinate (X) at (3.75,.5);
    \filldraw[fill=gray!30] (A) -- (D) -- (CD) -- (X) -- (AB) --cycle;
    \draw (A) -- (D) -- (E) -- (G) -- (H) -- (I) -- (A);
    \draw (A) -- (AB) -- (X) -- (CD) -- (F) -- (GH) -- (H) -- (IJ)
       (X) -- (IJ) -- (J) -- (AB) 
       (CD) -- (IJ) -- (GH) -- (CD) -- (C)
       (CD) -- (E) 
       (GH) -- (G)
       (CD) -- (B) -- (AB) -- (IJ) -- (I);
    \draw[dotted] (1,.5) .. controls +(1,.5) and +(-1,0) .. (3.2,.7)
                  .. controls +(1,0) and +(-1.2,2.4) .. (D);
    \foreach \p in {A,B,C,D,E,F,G,H,I,J}
      \draw (\p) node {\tiny$\bullet$};
  \end{tikzpicture}

\caption{Division with constraints along a facet.}
\label{fi:div}
\end{center}
\end{figure}
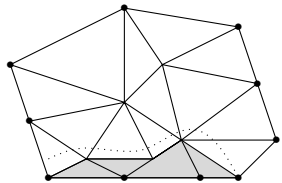

Although all these properties can be derived from the Barycentric Division in
$\RR^m$, none of them involves the notion of barycenter. The strength of
our Monohedral Division in $\Gamma^m$ (Theorem~\ref{th:mono-div}), and
eventually of our Monotopic Division in $K^m$
(Theorem~\ref{th:div-p-adic}), is that they preserve\footnote{In
  addition, the Monohedral and Monotopic Divisions even ensure that
  every $C$ in $\cC$ has no proper face inside $A$: either $C$ has no
  proper face at all, or its unique facet is outside $A$ (hence so are
all its proper faces) and belongs to $\cD$.} all of these properties.

\section{Notation, definitions}

\label{se:notation}

We let $\NN$, $\ZZ$, $\QQ$ denote respectively the set of non-negative integers,
of integers and of rational numbers, and $\NN^*=\NN\setminus\{0\}$. {\modif For every
$p,q\in\ZZ$ we let $[\![p,q]\!]$ denote the set of integers $k$ such that
$p\leq k\leq q$ (that is the empty set if $p>q$).}
\medskip

Recall that a {\df $\ZZ$\--group} is a linearly ordered group $G$ with a
smallest positive element such that $G/nG$ has $n$ elements for every
integer $n\geq1$. The reader unfamiliar with $\ZZ$\--groups may restrict to
the special but prominent case of $\ZZ$ itself. Indeed a linearly
ordered group is a $\ZZ$\--group if and only if it is elementarily
equivalent to $\ZZ$ (in the Presburger language $\cL_{Pres}$ defined below). 

$(K,v)$ is a $p$\--adically closed field in the sense of
\cite{pres-roqu-1984}, that is a Henselian valued field of
characteristic zero whose residue field is finite and whose value group
$\cZ=\Gamma\setminus\{+\infty\}=v(K^*)$ is a $\ZZ$\--group. A field is $p$\--adically closed
if and only if it is elementarily equivalent (in the language of
rings) to a finite extension of $\QQ_p$, so the reader unfamiliar with
the formalism of model-theory may restrict to this fundamental case. 

Let $\cQ$ be the divisible hull of $\cZ$. By identifying $\ZZ$ with the
smallest non-trivial convex subgroup of $\cZ$, we consider $\ZZ$
embedded into $\cZ$ (and $\QQ$ into $\cQ$). For every $a\in\cQ$
we let $|a|=\max(-a,a)$. 

{\modif 
  \begin{remark}\label{re:def-naive-1}
    When $\cZ=\ZZ$, one may naively define polytopes in $\ZZ^m$ as
    intersections $S\cap \ZZ^m$ where $S\subseteq\RR^m$ is a polytope of $\RR^m$. With
    other words, a polytope in $\ZZ^m$ (and more generally in $\cZ^m$)
    would be an intersection of finitely many half-spaces, that is the
    set of solutions of finitely many linear inequalities. Our
    polytopes are indeed so, but it will soon become clear that we need to
    be much more restrictive. Note first that such a definition does
    not lead naturally to a good notion of faces, because $S\cap\ZZ^m$ is
    closed with respect to the topology inherited from $\RR^m$. In
    particular if $S'$ is a face of $S$, $S'\cap\ZZ^m$ is disjoint from the
    closure of $S\cap\ZZ^m$. In order to carry significant topological
    properties, a notion of face for subsets of $\ZZ^m$ must involve
    points at infinity. 
  \end{remark}
}

For every $a$ in $\Omega=\cQ\cup\{+\infty\}$ we let $a+(+\infty)=(+\infty)+a=+\infty$. $\Omega$ is
endowed the topology generated by the open intervals and the intervals
$]a,+\infty]$ for $a\in\cQ$. $\Omega^m$ is equipped with the product topology, and
$\Gamma^m$ with the induced topology. The topological closure of any set
$A$ in $\Omega^m$ is denoted $\overline{A}$. Thus for example
$\Omega=\overline{\cQ}$ and $\Gamma=\overline{\cZ}$. The {\df frontier} of a
subset $A$ of $\Omega^m$ is the closure of $\overline{A}\setminus A$. We denote it
$\partial A$. 
\medskip

Whenever we take an element $a\in\Omega^m$ it is understood that
$a_1,\dots,a_m\in\Omega$ are its coordinates. We say that $a$ is {\df
non-negative} if all its coordinates are. A subset $A$ of $\Omega^m$ is
{\df non-negative} if all its elements are. {\modif A function $f$ with values
in $\Omega$ is {\df non-negative} (resp. {\df positive}) on a subset $X$ of
its domain if $f(x)$ is non-negative (resp. positive) for every $x\in
X$. }

If $m\geq1$ we let $\widehat{a}$ (resp. $\widehat{A}$)
denote the image of $a$ (resp. $A$) under the coordinate projection of
$\Omega^m$ onto the first $m-1$ coordinates $\Omega^{m-1}$. We call it the {\df
socle} of $a$ (resp. $A$). If $\cA$ is a family of subsets of $\Gamma^m$ we
also call $\widehat{\cA}=\{\widehat{A}\tq A\in\cA\}$ the socle of $\cA$. 

The {\df support of $a$}, denoted $\Supp a$, is the set of indexes $i$
such that $a_i\neq+\infty$. When all the elements of $A$ have the same
support, we call it the {\df support of $A$} and denote it $\Supp A$.
For every subset $I=\{i_1,\dots,i_r\}$ of $[\![1,m]\!]$ we let:
\begin{displaymath}
  F_I(A)=F_{i_1,\dots,i_r}(A)=\{a \in \overline{A}\tq \Supp a = I\}.
\end{displaymath}
When $F_I(A)\neq\emptyset$ we call it the {\df face of $A$ of support $I$}. It is
an {\df upward face} if moreover $m\in I$ and $F_{I\setminus\{m\}}(A)$ is
non-empty. Note that if $A$ is contained in $\Gamma^m$ then
so are its faces, because $\Gamma^m$ is closed in $\Omega^m$. By construction,
$F_I(A)=\overline{A}\cap F_I(\Omega^m)$ hence $\overline{A}$ is the disjoint
union of all its faces. A {\df complex} in $\Gamma^m$ is a finite family
$\cA$ of pairwise disjoint subsets of $\Gamma^m$ such that for every $A,B\in\cA$,
$\overline{A}\cap\overline{B}$ is the union of the {\em common} faces of $A$
and $B$. It is a {\df closed complex} if moreover it contains all the
faces of its members, or equivalently if $\bigcup\cA$ is closed. Note that
a finite partition $\cS$ of a subset $A$ of $\Gamma^m$ is a closed complex if
and only if $\cS$ contains the faces of all its members. 

The specialisation pre-order (see footnote~\ref{fo:spec-ord}) is an
order on the faces of $A$. The largest proper faces of $A$ with
respect to this order are called its {\df facets}. We say that $A$ is
{\df monohedral} if its faces are linearly ordered by specialisation.
Note that every subset of $F_I(\Gamma^m)$ is clopen in $F_I(\Gamma^m)$. In
particular if $A\subseteq F_I(\Gamma^m)$ then $\partial A$ is the disjoint union of its
proper faces. Note also that $F_J(A)=\emptyset$ whenever $J\nsubseteq I$.

\begin{example}\label{ex:face-non-cvx}
  Let $A\subseteq\ZZ^3$ be defined by $a_1\geq 0$, $a_2\geq a_1$ and $a_3=2a_2-2a_1$.
  It has four non-empty faces: $A$ itself, two facets
  $F_1(A)=\NN\times\{+\infty\}\times\{+\infty\}$ and $F_3(A)=\{+\infty\}\times\{+\infty\}\times2\NN$, plus
  $F_\emptyset(A)=\{(+\infty,+\infty,+\infty)\}$. 
\end{example}

We let $\pi^m_I$ be the natural projection of $\Gamma^m$ onto $F_I(\Gamma^m)$.
When $m$ is clear from the context, $\pi_I^m$ is simply 
denoted $\pi_I$. 

\begin{remark}\label{re:face-in-proj}
  For every $A\subseteq \Gamma^m$ note that $F_J(A)\subseteq\pi_J(A)$ and $\widehat{F_J(A)}\subseteq
  F_{\widehat{J}}(\widehat{A})$ (where $\widehat{J}=J\setminus\{m\}$). Indeed
  for every $b\in F_J(B)$ there are points in $A$ arbitrarily close to
  $b$. In particular there is a point $a\in A$ such that $\max_{i\in
  I}|a_i-b_i|<1$, which implies that $\pi_J(a)=b$. This proves the first
  inclusion, and we leave the second one to the reader. 
\end{remark}

For every $J\subseteq[\![1,m]\!]$ and $a\in\Omega^m$ we let $\Delta_J^m(a)=\min\{a_i\tq i\notin
J\}$ (if $J=[\![1,m]\!]$ we use the convention that
$\Delta^m_J(a)=\min\emptyset=+\infty$). Again the superscript $m$ is omitted whenever it
is clear from the context. Note that for every $a,b\in\Omega^m$
\begin{displaymath}
  \Delta_J(a+b) \geq \Delta_J(a) + \Delta_J(b).
\end{displaymath}

\begin{remark}\label{re:delta-dist}
  When $\cZ=\ZZ$ the topology\footnote{\modif The restriction of our topology
    to $\QQ^m$ agrees with the usual one, induced by the Euclidian
  distance.} on $\Omega^m$ comes from the distance $d(a,b)=\max_{1\leq i\leq m}
  |2^{-a_i}-2^{-b_i}|$, with the convention that $2^{-\infty}=0$. Thus
  $2^{-\Delta_J(a)}$ is just the distance from $a$ to its projection
  $\pi_J(a)$. In the general case the topology on $\Omega^m$ no longer comes
  from a distance. Nevertheless we will keep this geometric intuition in
  mind, that $\Delta_J(a)$ measures something like a distance from $a$ to
  $F_J(\Omega^m)$: the bigger $\Delta_J(a)$ is, the closer $a$ is to $F_J(\Omega^m)$.
\end{remark}

This intuitive meaning makes the following facts rather obvious. 

\begin{fact}\label{fa:delta-dist}
  For every function $f:A\subseteq \Gamma^m\to\Omega$, given $b\in\Gamma^m$ such that $\Supp b =J$  we
  have: 
  \begin{enumerate}
    \item
      $b\in F_J(A)$ iff $b\in\overline{A}$ iff 
      $\forall\delta\in\cZ$, $\exists a\in A$, $\pi_J(a)=b$ and $\Delta_J(a)\geq\delta$.
    \item 
      If $b\in F_J(A)$ then $f$ has limit $+\infty$ at $b$ iff $\forall\varepsilon\in\cZ$,
      $\exists\delta\in\cZ$, $\forall a\in A$, $[\pi_J(a)=b$ and $\Delta_J(a)\geq\delta] \Rightarrow f(a)\geq\varepsilon$. 
  \end{enumerate}
\end{fact}

Given a vector $u\in\cZ^m$ we let $A+u=\{x+u\tq x\in A\}$. We say
that $u$ is {\df pointing} to some $J\subseteq[\![1,m]\!]$ if $u_i=0$ for $i\in
J$ and $u_i>0$ for $i\notin J$. 

\begin{remark}\label{re:half-line}
  Let $J\subseteq I\subseteq[\![1,m]\!]$ and $S$ be any subset of $F_I(\Gamma^m)$. Using
  Remark~\ref{re:face-in-proj} and the above facts, one easily sees that
  if for every $\delta\in\cZ$ there is $u\in\cZ^m$ pointing to $J$ such that
  $\Delta_J(u)\geq\delta$ and $S+u\subseteq S$ then $F_J(S)=\pi_J(S)$, and in particular
  $F_J(S)\neq\emptyset$.
\end{remark}

{\modif
Example~\ref{ex:face-non-cvx} shows that even if a subset $A$ of $\ZZ^m$
is defined by finitely many linear inequalities, its faces may not be
so. Thus a polytope $A$ in $\Gamma^m$ must satisfy additional conditions,
in order to ensure that some linear inequalities which define $A$ also
define its faces after passing to the limits (see
Proposition~\ref{pr:pres-face} for a precise statement). It is 
these conditions that we are going to introduce now. }

A function $f:A\subseteq \Gamma^m\to\Omega$ is {\df largely continuous} on
$A$ if it can be extended to a continuous function on $\overline{A}$,
which we will usually denote $\bar f$. If $A$ has support $I$, we say
that $f$ is an {\df affine map} (resp. a {\df linear map}) if either $f$
is constantly equal to $+\infty$, or for some $\alpha_0\in\cQ$ (resp. $\alpha_0=0$) and
some $(\alpha_i)_{i\in I}\in\QQ^I$, we have
\begin{equation}\label{eq:def-lin}
  \forall a\in A,\quad f(a)=\alpha_0 + \sum_{i\in I}\alpha_ia_i.
\end{equation}
We call $\alpha_0$ the ``constant coefficient'' in the above expression of
$f$. If such an expression exists for which $\alpha_0\in\cZ$
and $\alpha_i\in \ZZ$ for $i\in I$, we say that $f$ is {\df integrally affine}. A
affine map which takes values in $\Gamma$ will be called {\df
$\Gamma$\--affine}. For example $f(x)=x/2$ is $\Gamma$\--affine on $2\ZZ$ but is
not integrally affine. 

\begin{remark}\label{re:linearity}
  Affinity and linearity are intrinsic properties because a function
  $\varphi:A\subseteq F_I(\Gamma^m)\mapsto\cQ$ is a linear map if and only if for every
  $a_1,\dots,a_k\in A$ and every $\lambda_1,\dots,\lambda_k\in\QQ^m$~:
  \begin{displaymath}
    \sum_{1\leq i\leq k}\lambda_ia_i\in A \ \Longrightarrow \ \varphi\bigg(\sum_{1\leq i\leq k}\lambda_ia_i\bigg)=\sum_{1\leq
    i\leq k}\lambda_i\varphi(a_i) 
  \end{displaymath}
\end{remark}

The symbols of the Presburger language
$\cL_{Pres}=\{0,1,+,\leq,(\equiv_n)_{n\in\NN^*}\}$ are interpreted as usually in
$\cZ$: the binary relation $a\equiv_nb$ says that $a-b\in n\cZ$, and the
other symbols have their obvious meanings. A subset $X$ of $\cZ^d$ is
{\df $\cL_{Pres}$\--definable} if there is a first order formula
$\varphi(\xi)$ in $\cL_{Pres}$, with parameters in $\cZ$ and a $d$\--tuple $\xi$
of free variables, such that $X=\{x\in\cZ^d\tq\cZ\models\varphi(x)\}$. A function
$f:X\subseteq\cZ^d\to\cZ$ is {\df $\cL_{Pres}$\--definable} if its graph is. 

Each $F_I(\Gamma^m)$ can be identified with $\cZ^d$ with $d=\Card(I)$. We say that a
subset $A$ of $\Gamma^m$ is {\df definable} if for every $I\subseteq[\![1,m]\!]$
the set $A\cap F_I(\Gamma^m)$ is $\cL_{Pres}$\--definable by means of this
identification. We say that a function $f:A\subseteq \Gamma^m\to\Omega$ is {\df definable}
if there is an integer $N\geq1$ such that $Nf(X)\subseteq\Gamma$ and if the
restrictions of $Nf$ to each $F_I(\Gamma^m)$ become, after this
identification, either an $\cL_{Pres}$\--definable map from
$\cZ^{\Card(I)}$ to $\cZ$ or the constant map $+\infty$. Note that every
affine map is definable in this broader sense. 

The next characterisation of definable maps and sets comes
directly from Theorem~1 in \cite{cluc-2003}. 

\begin{theorem}[Cluckers]\label{th:cluck-piece-lin} 
  For every definable function $f:A\subseteq \Gamma^m\to\Gamma$ on a non-negative set $A$,
  there exists a partition of $A$ in finitely many
  definable sets, on each of which the restriction of $f$ is
  an affine map. 
\end{theorem}

It is well known that the theory of $\ZZ$\--groups has quantifier
elimination and definable Skolem functions. At many places, without
mentioning, we will use the latter property under the following form.

\begin{theorem}[Skolem Functions]\label{th:skolem}
  Let $A\subseteq\cZ^m$ and $B\subseteq\cZ^n$ be two $\cL_{Pres}$\--definable sets.
  Let $\varphi(x,y)$ be a first order formula in $\cL_{Pres}$. If for every
  $a\in A$ there is $b\in B$ such that $\cZ\models\varphi(a,b)$ then there is a
  definable map $\lambda:A\to B$ such that $\cZ\models\varphi(a,\lambda(a))$ for every $a\in
  A$. 
\end{theorem}

Since $\cZ$ is elementarily equivalent to $\ZZ$ in the language
$\cL_{Pres}$, every non-empty $\cL_{Pres}$\--definable subset of $\cZ$
which is bounded above (resp. below) has a maximum (resp. minimum)
element. As a consequence for every $a\in\Omega$ there is in $\cZ$ a largest
element $\lfloor a\rfloor$ (resp. $\lceil a\rceil$) which is $\leq a$ (resp. $\geq a$). Note that
if $f:X\subseteq\cZ^d\to\cQ$ is definable and $N\geq1$ is an integer such that
$Nf$ is $\cL_{Pres}$\--definable, then for every integer $0\leq k<N$ the
set $S_k=\{x\in X\tq Nf(x)\equiv_N k\}$ is $\cL_{Pres}$\--definable, and so is
the map $\lfloor f\rfloor(x)=(Nf(x)-k)/N$ on $S_k$. Thus the map $\lfloor f\rfloor:S\to\cZ$ is
$\cL_{Pres}$\--definable, and so is $\lceil f\rceil$ by a symmetric argument.
Obviously the same holds true for every definable map from $A\subseteq\Gamma^m$
to $\Omega$. 

\begin{lemma}\label{le:precomp-borne}
  If $f:A\subseteq\Gamma^m\to\Omega$ is a largely continuous definable map on
  a non-negative set $A$, then it has a minimum in $A$. 
\end{lemma}

\begin{proof}
It suffices to prove the result separately for each $A\cap F_I(\Gamma^m)$ with
$I\subseteq[\![1,m]\!]$. Every such piece can be identified with a definable subset
of $\cZ^{\Card I}$ hence we can assume that $A\subseteq\cZ^m$. Multiplying $f$
by some integer $n\geq1$ if necessary we can assume that $f$ takes values
in $\Gamma$, and even in $\cZ$ (otherwise $f$ is constantly $+\infty$ and the
result is trivial). Since $\cZ\equiv\ZZ$, by instantiating the parameters of
a definition of  $f:A\subseteq\cZ^m\to\cZ$ it suffices to prove the result
for every largely continuous definable function on a non-negative subset
$A$ of $\ZZ^m$. But in that case the topology on $\Gamma^m$ comes from a
metric such that every non-negative subset of $\Gamma^m$ is precompact (that is
$\overline{A}$ is compact). So there is $\bar a\in\overline{A}$ such
that $\bar f(\bar a)=\min\{\bar f(x)\tq x\in\overline{A}\}$. For any $a\in
A$ close enough to $\bar a$ we have $f(a)=\bar f(\bar a)$ (because
$f(A)\subseteq\ZZ$) hence $f(a)=\min\{f(x)\tq x\in A\}$. 
\end{proof}

\begin{lemma}\label{le:image-compact}
  Let $f:A\subseteq\Gamma^m\to\Omega^n$ a continuous definable map. If $A$ is non-negative 
  then $f(\overline{A})$ is closed.
\end{lemma}

\begin{proof}
  As for Lemma~\ref{le:precomp-borne} we can reduce to the case where
  $\cZ=\ZZ$. But then $\overline{A}$ is compact, hence so is
  $f(\overline{A})$ since $f$ is continuous.
\end{proof}

We extend the binary congruence relations of $\cZ$ to $\Gamma$ with the convention
that $a\equiv+\infty\;[N]$ for every $a\in\Gamma$ and every $N\in\NN$. A subset $A$ of
$F_I(\Gamma^m)$ is a {\df basic Presburger set} if it is the set of
solutions of finitely many linear inequalities and congruence
relations. Although we will not use it, it is worth mentioning that,
by the quantifier elimination of the theory of $\ZZ$ in the language
$\cL_{Pres}$, the definable subsets of $\cZ^d$, and more generally of
$\Gamma^m$, are exactly the finite unions of basic Presburger sets.

Cluckers has shown in \cite{cluc-2003} that every definable subset of
$\cZ^d$ is actually the disjoint union of finitely many subsets of a
much more restrictive sort, called cells. The following definition of
precells in $\Gamma^m$, more precisely of precells mod $N$ for a given
$N\in(\NN^*)^m$, is adapted from his (see Remark~\ref{re:cell-cluc} below).
Since we only need to consider non-negative precells it is convenient to
restrict the definition to this case. If $m=0$, $\Gamma^0$ itself is the
only precell mod $N$ in $\Gamma^0$. If $m\geq1$, for every $I\subseteq[\![1,m]\!]$, a
subset $A$ of $F_I(\Gamma^m)$ is a {\df precell mod $N$} if: $\widehat{A}$ is
a precell mod $\widehat{N}$, and there are non-negative affine maps
$\mu,\nu:\widehat{A}\to\Omega$ and an integer $\rho$ such that $0\leq\rho<N_m$ and $A$ is
exactly the set of points $a\in F_I(\Gamma^m)$ such that
$\widehat{a}\in\widehat{A}$ and 
\begin{equation}
  \mu(\widehat{a})\leq a_m\leq \nu(\widehat{a}) \mbox{ \ and \ }a_m\equiv\rho\;[N_m].
  \label{eq:def-cell-mod-N}
\end{equation}
We call $\mu$, $\nu$ the {\df boundaries} of $A$, $\rho$ a {\df modulus} for $A$,
and such a triple $(\mu,\nu,\rho)$ a {\df presentation} of $A$. We call it a
{\df largely continuous presentation} of $A$ if moreover $\mu,\nu$ are
largely continuous. $A$ is a {\df largely continuous precell mod $N$} if
$m=0$, or if $\widehat{A}$ is largely continuous precell mod
$\widehat{N}$ and $A$ is a precell mod $N$ having a largely continuous
presentation.

\begin{remark}\label{re:cell-cluc}
  Cells in \cite{cluc-2003} are not required to have non-negative
  boundaries, but to be of one of these two types~: either $\mu-\nu$ is
  not finitely bounded or $\mu=\nu$. Unfortunately this condition seems to
  be too restrictive
  for our constructions, we have to relax it. Thus our precells are {\em
  not} cells in the sense of \cite{cluc-2003} but a restriction (we
  require non-negative boundaries) of a slight generalisation
  ($\mu-\nu$ can be finitely bounded and non-zero) of them.
\end{remark}

{\modif 
We are going to prove in the next section that the faces of every
largely continuous precell mod $N$ are still largely continuous
precells mod \emph{the same} $N$ (see Proposition~\ref{pr:pres-face}).
Thus, if one restricts to the case where $N=(1,\dots,1)$, these precells
are the best candidate we have for a discrete analogue of real
polytopes: they are intersections of finitely many half-spaces and
their faces are so. In the $p$\--adic triangulation that we are aiming
at we will indeed restrict to this case. However it appears that all
the results and constructions that we are going to consider in the
present paper remain valid for largely continuous precells mod
arbitrary $N$. Since it does not create significant complication, we
will then stick to this more general setting. 
}

\section{Faces and projections}
\label{se:face-proj}

In this section we consider a non-empty basic Presburger set $A\subseteq
F_I(\Gamma^m)$ defined by 
\begin{equation}
  \cconj_{1\leq l\leq l_0}\varphi_l(x)\geq\gamma_l \mbox{ \ and \ }
  \cconj_{1\leq l\leq l_1}\psi_l(x)\equiv\rho_l\;[n_l]
  \label{eq:basic-pres}
\end{equation}
where $\varphi_l,\psi_l:F_I(\Gamma^m)\to\cZ$ are integrally linear maps, $\gamma_l\in\cZ$, 
$\rho_l$ and $n_l$ are integers such that $0\leq\rho_l<n_l$. We prove some basic
properties on the faces of $A$ and the affine maps on $A$. Finally we
derive from these facts that every face of a largely continuous
precell $A$ mod $N$ is a largely continuous precell mod $N$ and has a
presentation inherited from $A$ in a uniform way (Proposition~\ref{pr:pres-face}). 
\smallskip

Example~\ref{ex:face-non-cvx} shows that precell mod $N$ (here
$N=(1,1,1)$) can have a facet which is no longer a precell mod $N$. But
even worse is possible: the next example shows that a precell mod $N$ can
have a facet which is not even a basic Presburger set. 

\begin{example}\label{ex:face-non-Presb}
  $A\subseteq\ZZ^3$ is defined by $0\leq x_1\leq x_2$ and $(x_1+3x_2)/3\leq z\leq
  (x_1+3x_2+1)/3$. Its unique facet $F_{1}(A)$ is defined by $0\leq x_1$
  and either $x_1\equiv0\;[3]$ or $x_1\equiv2\;[3]$ (and of course $x_2=x_3=+\infty$). 
\end{example}

\begin{lemma}\label{le:half-line}
  Let $A\subseteq F_I(\Gamma^m)$ be defined by (\ref{eq:basic-pres}). Let $J$ be any
  subset of $[\![1,m]\!]$. Then $F_J(A)\neq\emptyset$ if
  and only if for every $\delta\in\cZ$ there is $u\in\cZ^m$ pointing to $J$
  such that $\Delta_J(u)\geq\delta$ and $A+u\subseteq A$. 
\end{lemma}

\begin{proof}
It suffices to prove the result when $I=[\![1,m]\!]$. One direction is
general by Remark~\ref{re:half-line}, so let us prove the converse.
Assume that $F_J(A)\neq\emptyset$ and fix any $\delta\in\cZ$. Without lost of generality we can assume
that $\delta>0$. Pick $y_0\in F_J(A)$ and let $A_0=\{x\in A\tq \pi_J(x)=y_0\}$. By
Remark~\ref{re:face-in-proj}, $F_J(A_0)=\{y_0\}$. 

Assume that for some $k\in[\![0,l_0-1]\!]$ we have found a definable
subset $A_k$ of $A_0$ such that $F_J(A_k)=\{y_0\}$ and for every
$l\in[\![1,k]\!]$, either $\varphi_l$ is constant on $A_k$ or $\varphi_l(x)$ tends
to $+\infty$ as $x$ tends to $y_0$ in $A_k$. If the same holds true for
$\varphi_{k+1}$, let $A_{k+1}=A_k$. Otherwise, there is some $\alpha\in\cZ$ such that
for every $\omega\in\cZ$ there is $x\in A_k$ such that $\Delta_J(x)\geq\omega$ and
$\varphi_{k+1}(x)\leq\alpha$. The set $\Upsilon$ of these $\alpha$'s is definable, non-empty, and
bounded below since $\varphi_{k+1}\geq\gamma_{k+1}$ on $A_k$. Hence it has a
minimum, say $\beta$. By minimality of $\beta$ there is $\omega_0\in\cZ$ such that
for every $x\in A_k$ such that $\Delta_J(x)\geq\omega_0$, $\varphi_{k+1}(x)>\beta-1$. Thus, for
every $\omega\in\cZ$ there is $x\in A_0$ such that $\Delta_J(x)\geq\omega$ and $\varphi_{k+1}(x)=\beta$
(because $\beta\in\Upsilon$). With other words, the set $A_{k+1}$ defined by
\begin{displaymath}
  A_{k+1}=\big\{x\in A_k\tq \varphi_{k+1}(x)=\beta\big\} 
\end{displaymath}
is such that $F_J(A_{k+1})\neq0$ (see fact~\ref{fa:delta-dist}). It
obviously has all the other required properties since it is contained
in $A_k$. 

By repeating the process until $k=l_0$ we get a definable set
$A_{l_0}$ as above. Pick any $a\in A_{l_0}$, by construction there is
$\omega\in\cZ$ such that for every $x\in A_{l_0}$ if $\Delta_J(x)\geq\omega$ then $\varphi_l(x)\geq
\varphi_l(a)$ for every $l\in[\![1,l_0]\!]$. Pick any $b\in A_{l_0}$ such that
$\Delta_J(b)\geq\omega$ and $\Delta_J(b)\geq\delta+a_i$ for every $i\notin J$. It remains to check
that $u=b-a$ gives the conclusion. For every $j\in J$, $a_j=b_j=y_{0,j}$
because $a,b\in A_{l_0}\subseteq A_0$ and $\pi_J(A_0)=\{y_0\}$, hence $u_j=0$. For
$i\notin J$ we have $b_i\geq\Delta_J(b)\geq\delta+a_i$, hence $u_i\geq\delta>0$. In particular $u$
points to $J$ and $\Delta_J(u)\geq\delta$. 

Finally let $x$ be any element of $A_{l_0}$. For every $l\leq l_0$ we
have $\varphi_l(x)\geq\gamma_l$ since $x\in A$, and by linearity of $\varphi_l$
\begin{equation}
  \varphi_l(x+u)=\varphi_l(x)+\varphi_l(u)\geq\gamma_l+\varphi_l(u).
  \label{eq:fl-x-plus-u}
\end{equation}
We also have $\varphi_l(b)=\varphi_l(a)+\varphi_l(u)$ by linearity, and $\varphi_l(b)\geq \varphi_l(a)$
because $\Delta_J(b)\geq\omega$, hence $\varphi_l(u)\geq0$. It follows that $\varphi_l(x+u)\geq \gamma_l$ by
(\ref{eq:fl-x-plus-u}). On the other hand, for every $l\in[\![1,l_1]\!]$
we have $\psi_l(x)\equiv\rho_l\;[n_l]$ because $x\in A$, $\psi_l(a)\equiv\rho_l\;[n_l]$ and
$\psi_l(b)\equiv\rho_l\;[n_l]$ for the same reason, hence
$\psi_l(x+u)=\psi_l(x)+\psi_l(a)-\psi_l(b)\equiv\rho_l\;[n_l]$. Thus $x+u\in A$ for every
$x\in A$, which proves the result. 
\end{proof}

\begin{proposition}\label{pr:face-egale-proj}
  Let $A\subseteq F_I(\Gamma^m)$ be a basic Presburger set, $J$ and $H$ be any
  subsets of\/ $[\![1,m]\!]$ such that $F_J(A)$ and $F_H(A)$ are
  non-empty. 
  \begin{enumerate}
    \item\label{it:face-proj}
      $F_J(A)=\pi_J(A)$. 
    \item\label{it:face-face}
      If $H\subseteq J$ then $F_H(A)=F_H(F_J(A))$. 
    \item\label{it:mono-supp}
      $F_H(A)\subseteq\overline{F_J(A)}$ if and only if $H\subseteq J$. In particular
      the faces of $A$ are linearly ordered by specialisation if and
      only if their supports are linearly ordered by inclusion.
    \item\label{it:face-lattice}
      $F_{H\cap J}(A)$ is non-empty. 
  \end{enumerate}
\end{proposition}

We will refer to the $n$\--th point of
Proposition~\ref{pr:face-egale-proj} as to
Proposition~\ref{pr:face-egale-proj}($n$). 

\begin{remark}\label{re:sub-mono}
  Proposition~\ref{pr:face-egale-proj}(\ref{it:face-lattice}) shows
  that the set of faces of $A$ ordered by specialisation is a
  distributive lower semi-lattice with one smallest element. If $S$ is
  any monohedral subset of $\Gamma^m$,
  Proposition~\ref{pr:face-egale-proj}(\ref{it:mono-supp}) implies
  that every basic Presburger subset $A$ of $S$ is monohedral, and
  Proposition~\ref{pr:face-egale-proj}(\ref{it:face-face}) that every
  face of $A$ is again monohedral.
\end{remark}

\begin{proof}
The first point $F_J(A)=\pi_J(A)$ follows from Lemma~\ref{le:half-line}, by
Remark~\ref{re:half-line} applied to $S=A$.

For the second point, $H\subseteq J$ implies that $\pi_H(A)=\pi_H(\pi_J(A))$. Since
$F_H(A)=\pi_H(A)$ and $F_J(A)=\pi_J(A)$ by the first point, it suffices 
to prove that $F_H(\pi_J(A))=\pi_H(\pi_J(A))$. For every $\delta\in\cZ$ there is by
Lemma~\ref{le:half-line} a vector $u\in\cZ^m$ pointing to $H$ such
that $\Delta_H(u)\geq\delta$ and $A+u\subseteq A$. Then obviously
$\pi_J(A)+u=\pi_J(A+u)\subseteq\pi_J(A)$, and the conclusion follows by
Remark~\ref{re:half-line} applied to $S=\pi_J(A)$.

For the third point, one direction follows from the second point and
the other direction is general since $F_H(A)\subseteq
F_H(\Gamma^m)$, $\overline{F_J(A)}\subseteq\overline{F_J(\Gamma^m)}$, and $F_H(\Gamma^m)$ is
disjoint from $\overline{F_J(A)}$ if $H$ is not contained in $J$.

It remains to prove the last point. For every $\delta\in\cZ$,
Lemma~\ref{le:half-line} gives $u_J$ and $u_H$ in $\cZ^m$ pointing to
$J$ and $H$ respectively such that $\Delta_J(u_J)\geq\delta$, $A+u_J\subseteq A$ and
similarly for $u_H$. Without lost of generality we can assume that
$\delta>0$ hence for every $i\notin J\cap H$, $u_{J,i}+u_{H,i}\geq\delta>0$. In particular
$u_J+u_H$ points to $J\cap H$ and $\Delta_{J\cap H}(u_J+u_H)\geq\delta$. Obviously
$A+u_J+u_H$ is contained in $A$. So $F_{J\cap H}(A)$ is non-empty by
Remark~\ref{re:half-line}.
\end{proof}

\begin{proposition}\label{pr:prol-pi}
  Let $A\subseteq F_I(\Gamma^m)$ be a basic Presburger set defined by
  (\ref{eq:basic-pres}),  $f:A\to\Omega$ be an affine map, $J\subseteq I$ and
  $B=F_J(A)$. Assume that $B$ is not empty and that $f$ extends to a
  continuous map $f^*:A\cup B\to\Omega$. Then $f^*$ is affine, and if $f^*\neq+\infty$
  then $f=f^*_{|B}\circ{\pi_J}_{|A}$. In particular if $f^*\neq+\infty$ then
  $f(A)=f^*(B)$.
\end{proposition}

If $f$ is $\Gamma$\--affine then so is $f^*$ of course. However, if $f$ is
integrally affine we cannot conclude that $f^*$ will be integrally
affine as well, even if $f$ is largely continuous, as the following
example shows. 

\begin{example}\label{ex:prol-non-Z-lin}
  Keep $A\subseteq\ZZ^3$ as in Example~\ref{ex:face-non-cvx}. The map
  $f(x)=x_2-x_1$ is integrally affine and largely continuous on $A$,
  with $\overline{f}(x)=x_3/2$ on $\partial A$. This is no longer an
  integrally affine map on $B=F_3(A)=\{+\infty\}\times\{+\infty\}\times2\NN$.
\end{example}

\begin{proof}
It suffices to prove the result when $I=[\![1,m]\!]$, $f<+\infty$ is an
integrally linear map and $f^*$ is not constantly equal to $+\infty$.
{\modif Note
that $f^*$ is affine by Remark~\ref{re:linearity} (because equalities
satisfied by $f$ at every point of $A$ pass to the limits). So we only 
have to prove that $f(x)=f^*(\pi_J(x))$ for every $x\in A$. 
}

Let $\varphi$ be an integrally linear map on $\cZ^m$ extending $f$, and $b\in
B$ such that $f^*(b)<+\infty$. Since $f(A)\subseteq\cZ$ and $f(x)$ tends to
$f^*(b)$ at $b$, there exists $\delta\in\cZ$ such that for every $x\in A$, if
$\pi_J(x)=b$ and $\Delta_J(x)\geq\delta$ then $f(x)=f^*(b)$. Pick any $a\in A$ such
that $\pi_J(a)=b$ and $\Delta_J(a)\geq\delta$, hence $f(a)=f^*(b)$.

Now assume for a contradiction that $f(x_0)\neq f^*(\pi_J(x_0))$ for some
$x_0\in A$. Let $y_0=\pi_J(x_0)$, since $f(x)$ tends to $f^*(y_0)$ at
$y_0$ and $f^*(y_0)\neq f(x_0)$ there exists $\omega\in\cZ$ such that for every
$x\in A$, if $\pi_J(x)=y_0$ and $\Delta_J(x)\geq\omega$ then $f(x)\neq f(x_0)$.
Lemma~\ref{le:half-line} gives $u\in\cZ^m$ pointing to $J$ such
that $\Delta_J(u)\geq\omega-\Delta_J(x_0)$ and $A+u\subseteq A$. Then $\pi_J(x_0+u)=\pi_J(x_0)=y_0$
and $\Delta_J(x_0+u)\geq\Delta_J(x_0)+\Delta_J(u)\geq\omega$, hence $f(x_0+u)\neq f(x_0)$. By
linearity it follows that $\varphi(u)=f(x_0+u)-f(x_0)\neq0$. On the other hand
we have $\Delta_J(a+u)\geq\Delta_J(a)+\Delta_J(u)\geq\delta$ and $\pi_J(a+u)=\pi_J(a)=b$ hence
$f(a+u)=f^*(b)=f(a)$, and thus by linearity $\varphi(u)=f(a+u)-f(a)=0$, a
contradiction. 
\end{proof}

\begin{proposition}\label{pr:face-socle}
  Let $A\subseteq F_I(\Gamma^m)$ be a basic Presburger set with $m\geq1$, and
  $X=\widehat{A}$. Then for every face\footnote{\modif As already mentioned
    in Section~\ref{se:notation}, $A\subseteq F_I(\Gamma^m)$ and $F_J(A)\neq\emptyset$ imply
  that $J\subseteq I$.} $B=F_J(A)$, its socle
  $\widehat{B}=F_{\widehat{J}}(\widehat{A})$ is a face of
  $\widehat{A}$. If moreover $A$ is non-negative, then conversely for
  every face $Y$ of $X$ there is a face $B$ of $A$ such that
  $\widehat{B}=Y$. In that case $B=Y\times\{+\infty\}$ if $m\notin\Supp B$, and
  $B=(Y\times\cZ)\cap\overline{A}$ if $m\in\Supp B$.
\end{proposition}

\begin{remark}\label{re:face-socle}
  The last assertion on $B$ is general: for every subset $S$
  of $F_I(\Gamma^m)$ and every face $T=F_J(\Gamma^m)$ with socle $Y$, we have
  $T=Y\times\{+\infty\}$ if $m\notin J$, and $T=(Y\times\cZ)\cap\overline{S}$ if $m\in J$. Indeed
  $T=F_J(\Gamma^m)\cap\overline{S}$ and $F_J(\Gamma^m)$ is equal to
  $F_{\widehat{J}}(\Gamma^{m-1})\times\{+\infty\}$ if $m\notin J$ and to
  $F_{\widehat{J}}(\Gamma^{m-1})\times\cZ$ otherwise. 
\end{remark}

\begin{example}
  Let $A=\{x\in\ZZ^3\tq x_1-x_2-x_3=0\}$, its proper faces are
  $B_0=\{(+\infty,+\infty,+\infty)\}$, $B_1=\{+\infty\}\times\{+\infty\}\times\ZZ$ and $B_2=\{+\infty\}\times\ZZ\times\{+\infty\}$. Thus
  $\ZZ\times\{+\infty\}$ is a facet of $\widehat{A}=\ZZ^2$ which is not the socle
  of any face of $A$. 
\end{example}

This example shows that the assumption that $A$ is non-negative is
needed for the second part of Proposition~\ref{pr:face-socle} to
hold. Note that $\widehat{B}_1=\{(+\infty,+\infty)\}$ is not a facet of
$\widehat{A}=\ZZ^2$, which shows that the positivity of $A$
is mandatory also in Corollary~\ref{co:socle-facet}. 

\begin{proof} 
Given that $B=F_J(A)$ is a face of $A$, hence non-empty, let us prove
that $F_{\widehat{J}}(\widehat{A})=\pi_{\widehat{J}}(\widehat{A})$. For
every $\delta\in\cZ$ we can find a vector $u\in\cZ^m$ pointing to $J$ such
that $\Delta_J(u)\geq\delta$ and $A+u\subseteq A$, in particular $\widehat{u}$ points to
$\widehat{J}$, $\Delta_{\widehat{J}}(\widehat{u})\geq\delta$ and
$\widehat{A}+\widehat{u}\subseteq\widehat{A}$. Thus
$F_{\widehat{J}}(\widehat{A})=\pi_{\widehat{J}}(\widehat{A})$ by
Remark~\ref{re:half-line} applied to $S=\widehat{A}$. Since $B=\pi_J(A)$
by Proposition~\ref{pr:face-egale-proj}(\ref{it:face-proj}), and
obviously $\widehat{\pi_J(A)}=\pi_{\widehat{J}}(\widehat{A})$, it follows
that $\widehat{B}=F_{\widehat{J}}(\widehat{A})$. 

Now assume that $A$ is non-negative. Then the socle of $\overline{A}$ is
closed by Lemma~\ref{le:image-compact}. It contains $X$, hence
$\overline{X}$. In particular it contains $Y$, which is non-empty. So
there is $b\in\overline{A}$ whose socle $\widehat{b}$ belongs to $Y$.
Let $J=\Supp b$ and $B=F_J(A)$. Since $B$ contains $b$ it is
non-empty, hence a face of $A$. Then $\widehat{B}$  is a face of $X$
by the first point. Since $\widehat{b}$ belongs both to $Y$ and
$\widehat{B}$, it follows that $Y=\widehat{B}$. 
\end{proof}

\begin{corollary}\label{co:socle-facet}
  Let $A\subseteq F_I(\Gamma^m)$ be a non-closed non-negative basic Presburger set with
  socle $X$. Let $B$ be a facet of $A$ with socle $Y$. Then $Y=X$ or
  $Y$ is a facet of $X$. 
\end{corollary}

\begin{proof}
By Proposition~\ref{pr:face-socle}, $Y$ is a face of $X$. If $Y\neq X$
then there is a facet $Y'$ of $X$ whose closure contains $Y$. It
remains to show that $Y=Y'$. Proposition~\ref{pr:face-socle} gives a
face $B'$ of $A$ with socle $Y'$. Let $J$, $J'$, $H$, $H'$ be the
supports of $A$, $A'$, $B$, $B'$ respectively. Obviously $H=J\setminus\{m\}$ and
$H'=J'\setminus\{m\}$. If $B=B'$ then $J=J'$ hence $H=H'$ and thus $Y=Y'$. Now
assume that $B\neq B'$. Since $B$ is a facet of $A$ it is not smaller
than $B'$ (with respect to the specialisation order) hence $J\nsubseteq J'$ by
Proposition~\ref{pr:face-egale-proj}(\ref{it:mono-supp}). On the other
hand $Y\leq Y'$ hence $H\subseteq H'$ (otherwise $F_J(\Gamma^{m-1})$ is disjoint from
the closure of $F_{J'}(\Gamma^{m-1})$). Altogether this implies that
$J=J'\cup\{m\}$. In particular $H=J\setminus\{m\}=J'\setminus\{m\}=H'$ hence $Y=Y'$ is a facet
of $X$.
\end{proof}

\begin{proposition}\label{pr:pres-face}
  Let $A\subseteq F_I(\Gamma^m)$ be a largely continuous precell mod $N$ with $m\geq1$.
  Let $(\mu,\nu,\rho)$ be a largely continuous presentation of $A$, $J$ a
  subset of $I$, $\widehat{J}=J\setminus\{m\}$ and
  $Y=F_{\widehat{J}}(\widehat{A})$. Then $F_J(A)\neq\emptyset$ if and only
  if either $m\in J$ and $\bar\mu<+\infty$ on $Y$, or $m\notin J$ and $\bar\nu=+\infty$ on
  $Y$. In any case
  \begin{displaymath}
    F_J(A)=\big\{b\in F_J(\Gamma^m)\tq \widehat{b}\in Y,\  \bar\mu(\widehat{b})\leq
    b_m\leq \bar\nu(\widehat{b})\mbox{ and }b_m\equiv\rho\;[N_m]\big\}.
  \end{displaymath}
  In particular, if $F_J(A)$ is non-empty then it is a largely
  continuous precell $A$ mod $N$ and $(\bar\mu_{|Y},\bar\nu_{|Y},\rho)$ is a
  presentation of $F_J(A)$.
\end{proposition}

\begin{remark}\label{re:mono-cell}
  Combining the last point of the above result with
  Remark~\ref{re:sub-mono}, we get that if $A$ is a {\em monohedral}
  largely continuous precell mod $N$ in $\Gamma^m$ then so is every face of $A$. 
\end{remark}

\begin{proof}
Let $X$ be the socle of $A$. Recall that $Y=F_{\widehat{J}}(X)$ is a
face of $X$ and of the socle of $F_J(A)$ by
Proposition~\ref{pr:face-egale-proj}(\ref{it:face-proj}). Let $B$ be
the set of $a\in F_J(\Gamma^m)$ such that $\widehat{a}\in
F_{\widehat{J}}(\widehat{A})=Y$, $\bar\mu(\widehat{a})\leq a_m\leq
\bar\nu(\widehat{a})$ and $a_m\equiv\rho\;[N_m]$. Non-strict inequalities and
congruence relations valid on $A$ pass to the limits, hence remain
valid on $F_J(A)$. So $B\subseteq F_J(A)$, and if $F_J(A)\neq\emptyset$ then necessarily
one of the two alternatives of the first point hold true. 

Conversely, take any point $a\in A$ and let $b=\pi_J(a)$. Assume first
that $m\in J$ and $\bar\mu<+\infty$. By Proposition~\ref{pr:prol-pi},
$\bar\mu\circ\pi_J=\mu$ on $X$ hence $\bar\mu(\widehat{b})=\mu(\widehat{a})\leq
a_m=b_m$. If  $\bar\nu<+\infty$ then similarly $b_m\leq\bar\nu(\widehat{b})$.
Otherwise $\bar\nu=+\infty$ and $b_m\leq\bar\nu(\widehat{b})$ is obvious. Since
$b_m=a_m\equiv\rho\;[N]$ it follows in both cases that $b\in B$. Now assume
that $m\notin J$ and $\bar\nu=+\infty$. Then $b_m=+\infty$, hence obviously
$\bar\mu(\widehat{b})\leq +\infty=b_m=\bar\nu(\widehat{b})$ and $b_m=+\infty\equiv\rho\;[N]$.
Thus $b\in B$, which proves that $\pi_J(A)\subseteq B$. In particular $B\neq\emptyset$,
hence $F_J(A)\neq\emptyset$ since it contains $B$. This proves the first point.
Moreover by Proposition~\ref{pr:face-egale-proj}(\ref{it:face-proj})
it follows that $F_J(A)=\pi_J(A)\subseteq B$. Hence $F_J(A)=B$, which proves the
second point. In particular $F_J(A)$ is a largely continuous precell if
$F_{\widehat{J}}(Y)$ is so. The remaining of the conclusion then
follows by a straightforward induction. 
\end{proof}

\section{Bounding functions}
\label{se:bound-fun}

We prove here two technical results (Propositions~\ref{pr:maj-Z-aff}
and \ref{pr:min-Z-aff}) used in the next section. 

\begin{proposition}\label{pr:maj-Z-aff}
  Let $A\subseteq F_I(\Gamma^m)$ be a definable set, and $f_1,\dots,f_r$ be definable
  maps from $A$ to $\cQ$. Assume that the coordinates of all the
  points of $A$ are non-negative. Then there exists a largely continuous,
  positive, integrally affine map $f:A\to\cZ$ such that $f(x)\geq
  \max_j f_j(x)$ on $A$ and $\bar f=+\infty$ on $\partial A$. More precisely $f$
  can be taken of the form $f(x)=\beta+\alpha\sum_{i\in I}x_i$ on $A$, for some
  positive $\alpha\in\ZZ$ and $\beta\in\cZ$.
\end{proposition}

\begin{proof}
Without lost of generality we can assume that $I=[\![1,m]\!]$.
Because of Theorem~\ref{th:cluck-piece-lin}, it suffices to consider
the case where the $f_j$'s
are affine. Each $f_j$ then can be written as
\begin{displaymath}
  f_j(x)=\alpha_{0,j}+\sum_{1\leq i\leq m} \alpha_{i,j}x_i
\end{displaymath}
for some $\alpha_{i,j}\in\ZZ$ for $i\geq1$ and some $\alpha_{0,j}\in\cZ$. Let $\alpha\geq1$ be an
integer greater than $\alpha_{i,j}$ for every $i,j\geq1$, and $\beta\geq1$ an element
of $\cZ$ greater the $\alpha_{0,j}$ for every $j\geq1$. For every $x$ in $A$
and every $i,j\geq1$, since $x_i\geq0$ we have $\alpha x_i\geq \alpha_{i,j}x_i$. So the
function $f(x)=\beta+\alpha\sum_{1\leq i\leq m} x_i$ has all the required properties. 
\end{proof}

\begin{lemma}\label{le:f-hat-C0bar}
  Let $A\subseteq\cZ^m$ be a largely continuous precell mod $N$, $X$ its socle,
  $(\mu,\nu,\rho)$ a largely continuous presentation of $A$ and $f$ a largely
  continuous affine map on $A$ such that $\bar f=+\infty$ on $\partial A$. Let
  $(\alpha_i)_{1\leq i\leq m}\in\QQ^m$ and $\beta\in\cQ$ be such that $f(a)=\beta+\sum_{1\leq i\leq
  m}\alpha_i a_i$ on $A$. Extend $f$ to $\cQ^m$ by means of this expression.
  For every $x\in\widehat{A}$ let $\hat f(x)=f(x,\mu(x))$ if $\alpha_m\geq0$, and
  $\hat f(x)=f(x,\nu(x))$ otherwise. Then $\hat f$ is a well-defined
  largely continuous affine map on $X$ with limit $+\infty$ at every point of $\partial
  X$, and $\min f(A)-|\alpha_m|N_m\leq \hat f(\widehat{a})\leq f(a)$ for every $a\in A$. 
\end{lemma}

\begin{proof}
The only possible problem in the definition of $\hat f$ is when
$\nu=+\infty$. But then $\alpha_m\geq0$ because otherwise, given any $x\in X$ we have
$(x,+\infty)\in\partial A$ and $f(a)<0$ for every $x\in A$ close enough to $(x,+\infty)$, a
contradiction since $\bar f=+\infty$ on $\partial A$. Thus $\hat f(x,\mu(x))$ is
well-defined in this case too.

Let $\lambda=\mu$ if $\alpha_m\geq0$ and $\lambda=\nu$ otherwise. Then $\hat f(x)=f(x,\lambda(x))$
is an affine map and $\alpha_m(a_m-\lambda(\widehat{a}))$ is non-negative on $A$ by
construction, hence
\begin{displaymath}
  f(a) = f\big(\widehat{a},\lambda(\widehat{a})\big)
         + \alpha_m\big(a_m-\lambda(\widehat{a})\big)
         \geq \hat f(\widehat{a}). 
\end{displaymath}

For every $x\in X$ there is a point $a\in A$ such that
$\widehat{a}=x$ and $|a_m-\lambda(x)|\leq N_m$. So there is a definable
function $\delta:X\to\cZ$ such that $(x,\delta(x))\in A$ and
$|\delta(x)-\lambda(x)|\leq N_m$ for every $x\in X$. We have
\begin{eqnarray*}
  f(x,\lambda(x))&=& f\big(x,\delta(x)\big)+\alpha_m\big(\lambda(x)-\delta(x)\big) \\
  &\geq& f\big(x,\delta(x)\big)-|\alpha_m|N_m \label{eq:ineg-f-crochet}
\end{eqnarray*}
In particular $\hat f(x)\geq \min f(A)-\alpha_mN_m$ on $X$. 

It only remains to check that, given any $y\in \partial X$, $f(x,\lambda(x))$ tends
to $+\infty$ when $x\in A$ tends to $y$. By the above inequality it suffices
to prove that $f(x,\delta(x))$ tends to $+\infty$ when $x\in A$ tends to $y$.
Since $(x,\delta(x))\in A$ for every $x\in X$ and $\bar f=+\infty$ on $\partial A$, it
is sufficient to show that $\delta(x)$ tends to a limit $l\in\Gamma$ as $x\in X$
tends to $y$. Indeed, since $y\in\partial X$ we will then have that $(x,\delta(x))$
tends to $(y,l)\in\partial A$ so the conclusion. We prove it only when $\alpha_m\geq0$,
the case where $\alpha_m<0$ being similar. 

If $\bar\mu(y)=+\infty$, then obviously $\delta(x)$ tends to $+\infty$ since
$\mu(x)\leq\delta(x)$. If $\bar\mu(y)<+\infty$ then $\mu(x)=\bar\mu(y)$ for every $x\in X$
close enough to $y$. Hence $\delta(x)$, which is the smallest element $t$ in
$\Gamma$ such that $\mu(x)\leq t$ and $t\equiv\rho\;[N_m]$, remains constant too. In
particular it has a limit in $\cZ$ as $x\in X$ tends to $y$. 
\end{proof}

\begin{proposition}\label{pr:min-Z-aff}
  Let $A\subseteq F_I(\Gamma^m)$ be a largely continuous precell mod $N$, and
  $f_1,\dots,f_r$ be largely continuous affine maps on $A$ such
  that $\overline{f_j}=+\infty$ on $\partial A$ for every $j$. Then there exists a
  largely continuous affine map $f$ on $A$ such that
  $\overline{f}=+\infty$ on $\partial A$ and $f(x) \leq \min_j f_j(x)$ for every $x\in
  A$. If, moreover, each $f_j$ is positive on $A$ then $f$
  can be chosen positive on $A$. 
\end{proposition}

\begin{proof}
W.l.o.g. we can assume that $A\subseteq \cZ^m$ and $f_j<+\infty$ for every $j$. By
Lemma~\ref{le:precomp-borne} there is $\gamma\in\cQ$ such that
$\gamma=\min\bigcup_jf_j(A)$. Given an arbitrary $\gamma'<\gamma$ in $\cQ$ we are going to
show that there exists a largely continuous map $f:A\to\cQ$ such that
$\bar f=+\infty$ on $\partial A$ and $\gamma'\leq f(x)\leq\min_j f_j(x)$ on $A$. This will
prove simultaneously the two statements, because if each $f_j$ is
positive then $\gamma>0$ hence taking for example $\gamma'=\gamma/2$ will give
that $0<\gamma/2\leq f$ on $A$.

The proof goes, needless to say, by induction on $m$. If $m=0$, and
more generally if $A$ is closed, the constant function $f=\gamma$ has the
required properties. So we can assume that $A$ is not closed, $m\geq1$
and the result is proven for smaller integers. Replacing each $f_j$ by
$f_j-\gamma$ we can assume that $\gamma=0$. Replacing $\gamma'<0$ by a bigger one if
necessary we can assume that $\gamma'\in\QQ$.

Let $\alpha_{i,j}\in\QQ$ and $\beta_j\in\cQ$ such that $f_j(x)=\beta_j+\sum_{1\leq i\leq
m}\alpha_{i,j}x_i$. Let $\hat f_j:X\to\cQ$ be defined as in
Lemma~\ref{le:f-hat-C0bar}, and $\eta=\min\bigcup_j\hat f_j(X)$. By
Lemma~\ref{le:f-hat-C0bar} the induction hypothesis applies to these
functions. Given any $\eta'<\eta$, it gives a largely continuous affine map
$g:X\to\cQ$ such that $\bar g=+\infty$ on $\partial X$ and $\eta'\leq g(x)\leq\hat f_j(x)$ on
$X$ for $1\leq j\leq r$. We do this for $\eta'=-(\max_j|\alpha_{m,j}|N_m+1)$. Indeed
by Lemma~\ref{le:f-hat-C0bar}, $-|\alpha_{m,j}|N_m\leq\hat f_j$ on $X$ for $1\leq
j\leq r$ hence $\eta'\leq\eta-1<\eta$. Since $\eta'<0$, replacing $\gamma'$ by a bigger one if necessary we
can assume that $\eta'\leq\gamma'$. 

\paragraph{Case 1:} $\nu_A=+\infty$. Then for $1\leq j\leq r$ the coefficient
$\alpha_{m,j}$ of $x_m$ in the above expression of $f_j$ is 
positive (see the proof of lemma~\ref{le:f-hat-C0bar}), hence $\hat
f_j(x)=f_j(x,\mu(x))$ and $\alpha=\min_{j\leq r}\alpha_{m,j}$ is positive.
Let $G(a)=g(\widehat{a})+\alpha(a_m-\mu(\widehat{a}))$ on $A$. For $1\leq j\leq r$
we have
\begin{displaymath}
  G(a) \leq \hat f_j(\widehat{a}) + \alpha_{m,j}\big(a_m-\mu(\widehat{a})\big)
       = f_j\big(\widehat{a},\mu(\widehat{a})\big)
         + \alpha_{m,j}\big(a_m-\mu(\widehat{a})\big)
       = f_j(a).
\end{displaymath}
Every $b\in\partial A$ either belongs to $X\times\{+\infty\}$ or to $\partial X\times\Gamma$. If $b\in X\times\{+\infty\}$
then $G(a)=g(\widehat{a})+\alpha(a_m-\mu(\widehat{a}))$ tends to $+\infty$ as $a\in
A$ tends to $b$, because $\widehat{a}$ then tends to $\widehat{b}$,
$a_m$ tends to $+\infty$ and $\alpha>0$. If $b\in \partial X\times\Gamma$ then $G(a)\geq
g(\widehat{a})$ tends to $+\infty$ as $a\in A$ tends to $b$, because
$\widehat{a}$ then tends to $\widehat{b}$. Hence $G$ is largely
continuous and $\bar G=+\infty$ on $\partial A$. 

\paragraph{Case 2:} $\nu_A<+\infty$. Then  every $b\in \partial A$ belongs to $\partial X\times\Gamma$
hence $g(\widehat{a})$ tends to $+\infty$ as $a\in A$ tends to $b$. Moreover
$g(\widehat{a})\leq\hat f_j(\widehat{a})\leq f_j(a)$ for $1\leq j\leq r$. 

\paragraph{Cases 1 and 2:} In both cases, it remains to modify $G$ so
that its minimum becomes greater than $\gamma'$. By construction $G(a)\geq
g(\widehat{a})\geq \eta'$ on $A$. Recall that $\eta'=-(\max_j|\alpha_{m,j}|N_m+1)$
and $\gamma'\geq \eta'$ are strictly negative {\em rational} numbers. Thus we
can define $f(a)=(\gamma'/\eta') G(a)$ on $A$. Clearly $f$ is a largely
continuous affine function on $A$ with $\bar f=+\infty$ on $\partial A$, and $f\geq
(\gamma'/\eta')\eta'=\gamma'$ since $\gamma'/\eta'\geq0$ and $G\geq\eta'$ on $A$.  Moreover $0\leq\gamma'/\eta'\leq
1$ hence for every $a\in A$:
\begin{displaymath}
  f(a)=\frac{\gamma'}{\eta'}G(a)\leq\max(0,G(a))\leq\min_{1\leq j\leq r}f_j(a)
\end{displaymath}
\end{proof}

\section{Monohedral division}
\label{se:division}

{\modif

The next lemma is the technical heart of this paper. Loosely speaking,
given a precell $A\subseteq\Gamma^m$, a facet $B$ of $A$, a function $f:B\to\cZ$ and
a family $\cD$ of monohedral precells covering $B$, we are going to
inflate each $D$ in $\cD$ to a precell $C_D\subseteq A$ in such a way that: 
\begin{enumerate}
  \item 
    $C_D$ is a monohedral precell with facet $D$;
  \item
    the shape of $C_D$ is controlled by $f$, in the sense that the
    distance to $B$ of any point $a\in C_D$ is less than $f(b)$ (where
    $b$ is the projection of $a$ onto $B$);
  \item 
    $C_D$ contains a ``neighbourhood of $D$'', so to say, in the sense
    that every point of $A$ close enough to $D$ belongs to $C_D$ (the
    ``close enough'' condition will be controlled by a function
    $\delta:B\to\cZ$);
  \item 
    the various $C_D$'s do not intersect too much (in particular if
    $\cD$ is a partition of $B$, we require the various precells $C_D$ to
    be pairwise disjoint).
\end{enumerate}
In addition we construct simultaneously a family $\cU$ of precells
partitioning the complement of $\bigcup_{D\in\cD}C_D$ in $A$, such that the
proper faces of every $U\in\cU$ are proper faces of $A$ different from
$B$. In particular $U$ has less faces than $A$, which will make
possible to repeatedly use the next lemma (first applied to $A$, then
to each $U\in\cU$) while proving results by induction on the number of faces.

}

\begin{lemma}\label{le:face-elargie}
  Let $A\subseteq F_I(\Gamma^m)$ be a non-closed largely continuous precell
  mod $N$. Let $B$ be a facet of
  $A$, $J$ its support, $f:B\to\cZ$ a definable map. Let $\cD$ be a
  family of largely continuous monohedral precells mod $N$ such that
  $\bigcup\cD= B$. Then there exists a pair $(\cC,\cU)$ of families of
  largely continuous precells mod $N$ contained in $A$ and an integrally
  affine map $\delta:B\to\cZ$ such that $\cU$ is a finite partition of
  $A\setminus\bigcup\cC$, the proper faces of every precell in $\cU$ are proper faces
  of $A$, and $\cC$ is a family $(C_D)_{D\in\cD}$ of precells with the
  following properties:
  \begin{description}
    \item
      [(Fac)] $C_D$ has a unique facet which is $D$.
    \item
      [(Sub)] $C_D\subseteq\{a\in A\tq \pi_J(a)\in D$ and $\Delta_J(a)\geq f\circ\pi_J(a)\}$.
    \item
      [(Sup)] $C_D\supseteq\{a\in A\tq \pi_J(a)\in D$ and $\Delta_J(a)\geq \delta\circ\pi_J(a)\}$.
    \item
      [(Diff)] For every $E\in\cD$, $\pi_J(C_D\setminus C_E)\subseteq D\setminus E$.
  \end{description}
\end{lemma}

\begin{remark}\label{re:face-elargie}
  In every application of Lemma~\ref{le:face-elargie}, $\cD$ will be a
  partition of $B$. So the condition (Diff) simply says that the precells
  in $\cC$ are pairwise disjoint, hence that $\cC\cup\cU$ is a
  partition of $A$. However we can not restrict to this case because
  it may happen that $\cD$ is a partition of $B$ and $\widehat{\cD}$
  is not a partition of $\widehat{B}$, which will be crippling when
  proving the result by induction on $m$.
\end{remark}

Before entering in the somewhat intricate proof of this lemma, let us
make a few preliminary observations. 

\begin{claim}\label{cl:U-pas-B}
  With the notation of Lemma~\ref{le:face-elargie}, $B$ is not a face
  of any $U\in\cU$.
\end{claim}

\begin{proof}
For every $b\in B$ there is $D\in\cD$ such that $b\in C_D$. By (Sup) every
point in $A$ such that $\pi_J(a)=b$ and $\Delta_J(a)\geq\delta\circ\pi_J(a)$ belongs to
$C_D$, hence not to $U$. Thus $b\notin\overline{U}$, that is
$B\cap\overline{U}=\emptyset$.
\end{proof}

\begin{claim}\label{cl:fac}
  Let $A\subseteq F_I(\Gamma^m)$ be a non-closed largely continuous precell mod
  $N$, $B$ a facet of $A$, $J=\Supp B$. Let $C_D$, $D$ be any precells
  mod $N$ contained in $A$, $B$ respectively, satisfying conditions
  (Sub) and (Sup) of Lemma~\ref{le:face-elargie} for some definable
  maps $f,\delta:B\to\cZ$. If $f$ is largely continuous and $\bar f=+\infty$ on $\partial
  B$ then property (Fac) of Lemma~\ref{le:face-elargie} follows: $C_D$
  has a unique facet which is $D$. 
\end{claim}

\begin{proof}
  For every $b\in D$ and every $\varepsilon\in\cZ$, $b\in\overline{A}$ hence there
  exists $a\in A$ such that $\pi_J(a)=b$ and $\Delta_J(a)\geq\max(\delta(b),\varepsilon)$. By
  (Sup) this point $a$ belongs to $C_D$, hence $b$ is in the closure
  of $C_D$. So $D\subseteq F_J(C_D)$, and conversely (Sub) implies that
  $\pi_J(C_D)\subseteq D$, hence $F_J(C_D)=\pi_J(C_D)=D$ by
  Proposition~\ref{pr:face-egale-proj}(\ref{it:face-proj}). 
  
  Assume for a contradiction that $C_D$ has a proper face $F_H(C_D)$
  not contained in $\overline{D}$. Pick any $c$ in $F_H(C_D)$. By
  Proposition~\ref{pr:face-egale-proj}(\ref{it:mono-supp}), $H$ is not
  contained in $J$ so pick any $k\in H\setminus J$. By
  Proposition~\ref{pr:face-egale-proj}(\ref{it:face-lattice}), $F_{J\cap
  H}(C_D)\neq\emptyset$ hence by the remaining of
  Proposition~\ref{pr:face-egale-proj}, $\pi_{J\cap H}(C_D)=F_{J\cap
  H}(C_D)=F_{J\cap H}(F_H(C_D))\subseteq F_{J\cap H}(B)\subseteq \partial B$. So $\pi_{J\cap H}(c)\in \partial
  B$, hence $f$ has limit $+\infty$ at $\pi_{J\cap H}(c)$. In particular there
  is $\delta\in\cZ$ such that for every $b\in B$
  \begin{equation}
    \big[\pi_{J\cap H}(b)=\pi_{J\cap H}(c)\mbox{ and }\Delta_{J\cap H}(b)\geq \delta\big]
    \Rightarrow f(b)>c_k
    \label{eq:fac1}
  \end{equation}
  On the other hand $c\in F_H(C_D)$ hence there is $a\in C_D$ such that
  $\pi_H(a)=c$ and $\Delta_H(a)\geq\delta$. Let $b=\pi_J(a)$, then $\pi_{J\cap H}(b)=\pi_{J\cap
  H}(a)=c$ and
  \begin{displaymath}
    \Delta_{J\cap H}(b) = \min_{j\notin H}b_j = \min_{j\in J\setminus H}a_j 
    \geq \min_{i\notin H}a_i = \Delta_H(a)\geq\delta. 
  \end{displaymath}
  By (\ref{eq:fac1}) this implies that $f(b)>c_k$, that is
  $f\circ\pi_J(a)>c_k$. By (Sub) it follows that $\Delta_J(a)>c_k$, a
  contradiction since $\Delta_J(a)=\min_{j\notin J}a_j\leq a_k$ (because $k\notin J$)
  and $a_k=c_k$ (because $k\in H$ and $\pi_H(a)=c$. 
\end{proof}

\begin{proof}[(of Lemma~\ref{le:face-elargie})]
Let $(\mu,\nu,\rho)$ be a largely continuous presentation of $A$. Let $X$,
$Y$ be the socles of $A$, $B$ respectively, $\widehat{I}=\Supp X$ and
$\widehat{J}=\Supp Y$. Since $B$ is a facet of $A$, by
Proposition~\ref{pr:face-socle} either $Y=X$ and $B=X\times\{+\infty\}$, or $Y$ is
a facet of $A$ and either $B=Y\times\{+\infty\}$ or $B=(Y\times\cZ)\cap\overline{A}$. For
each $D\in\cD$ let $(\mu_D,\nu_D,\rho_D)$ be a largely continuous presentation
of $D$. 

If $m=0$ the result is trivially true because there is no non-closed
precell contained in $\Gamma^0$. So we can assume that $m\geq1$ and the result
is proven for smaller integers. If $\mu=+\infty$ then $A=X\times\{+\infty\}$ can be
identified with $X$, after which the result follows by the induction
hypothesis. So we can assume that $\mu<+\infty$. 
Proposition~\ref{pr:maj-Z-aff} gives positive $\alpha\in\ZZ$ and
$\beta\in\cZ$ such that $f(x)\leq\beta+\alpha\sum_{j\in J}x_j$ on $B$. Without loss of
generality we can assume that equality holds on $B$, and we still
denote by $f$ the corresponding extension of $f$ to $F_J(\Omega^m)$. 
In particular $f$ is now largely continuous on $B$ with $\bar f_B=+\infty$
on $\partial B$. 

{\modif
It is sufficient  to build a pair $((C_D)_{D\in\cD},\cU)$ of families of
largely continuous precells mod $N$ contained in $A$ and a definable
map $\delta:B\to\cQ$ such that $\cU$ is a finite partition of $A\setminus\bigcup\cC$, that
the proper faces of every precell in $\cU$ are proper faces of $A$,
and that for each $D$ in $\cD$ we have:
\begin{itemize}
  \item[(Fac')]
     $\pi_J(C_D)=D$;
  \item[(Sub')] 
    $\Delta_J\geq f\circ\pi_J$ on $C_D$;
  \item [(Sup)]
    $C_D\supseteq\{a\in A\tq \pi_J(a)\in D$ and $\Delta_J(a)\geq\delta(\pi_J(a))\}$;
  \item[(Diff')]
    $\pi_J(C_D\setminus C_E)$ is disjoint from $E$, for every $E\in\cD$.
\end{itemize}
Indeed, we do not need to require that $\delta$ is integrally affine,
because if property (Sup) holds for a definable map $\delta:B\to\cQ$ it will
then hold for every larger map, and Proposition~\ref{pr:maj-Z-aff}
provides an integrally affine one. Then by (Fac'), properties (Sub)
and (Diff) will follow from (Sub') and (Diff'). Because $\bar f=+\infty$ on
$\partial B$, (Fac) will then follow from (Sup) and (Sub) by
Claim~\ref{cl:fac}. 

We have to distinguish several cases. For the convenience of the
reader most of them are accompanied by a figure representing (very
approximatively) the general idea of the construction when $m=2$. In
these figures, each $(i,j)\in\Gamma^2$ with positive coordinates takes place
at the point of coordinates $(1-2^{-i},1-2^{-j})$ in the figure, so
that the order is preserved. Therefore the set of positive points of
$\Gamma^2$ is represented by a square whose bottom, left, right and
top edges represent $\Gamma\times\{0\}$, $\{0\}\times\Gamma$, $\Gamma\times\{+\infty\}$ and $\{\infty\}\times\Gamma$
respectively\footnote{The reader may imagine that this lower edge of
  the square is actually $\Gamma^{m-1}$ in the induction steps, while the
left edge is $\Gamma$.}\;. A side effect of this compactification is that
linear functions are represented by curved lines. 

The various precells involved will be represented in this square by
gray areas whose union is $A$ (or the auxiliary precell $A^\circ$ in
figure~\ref{fi:div-3-right}). Their socles will take place in the
bottom $\Gamma\times\{0\}$, as we identify it with $\Gamma$. Finally $B$ will be
represented by a thick edge or a corner of the square, depending on
the cases. 

}

\paragraph{Case 1:} $Y=X$. 
\\

Then $B=X\times\{+\infty\}$ hence $\nu=+\infty$ and $J= I\setminus\{m\}$, thus $\Delta_J(a)=a_m$ and
$\pi_J(a)=(\widehat{a},+\infty)$ for every $a\in A$. So $B$ identifies to $X$
and $\cD$ to $\widehat{\cD}$. Roughly speaking, we are going on one
hand to split $A$ in two parts by means of a function $\lambda$ to be
defined such that $\mu<\lambda$ and $f<\lambda$, and on the other hand to lift the
family $\widehat{\cD}$ of $X$ to a family of precells covering of the
upper part of $A$. This will give us $\cC$. As figure~\ref{fi:div-1}
suggests, the lower part of $A$ will remain unchanged and give $\cU$. 

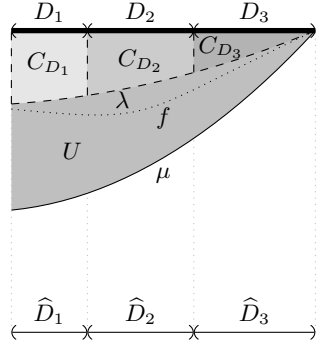
\begin{figure}[h]
  \small
  \begin{center}
    \begin{tikzpicture}[scale=.4]

      \def\fonctionL#1{plot[domain=#1] (\x,{2.5*((\x+2)/12)^2+7.5})}
      \def\fonctionMU#1{plot[domain=#1] (\x,{6*((\x+1)/11)^2+4})}

      \coordinate (A) at (0,7.5694444444444);
      \coordinate (B) at (2.5,7.85156250000000);
      \coordinate (C) at (6,8.611111111111111);

      % Cellules U (en fait A), C_D1, C_D2, C_D3
      \fill[color=gray!50] \fonctionMU{0:10} -- (0,10) --cycle;
      \fill[color=gray!20] \fonctionL{0:2.5} -- (2.5,10) -- (0,10) --cycle; 
      \fill[color=gray!40] \fonctionL{2.5:6} -- (6,10) -- (2.5,10) --cycle; 
      \fill[color=gray!60] \fonctionL{6:10} -- (6,10) --cycle; 

      \draw \fonctionMU{0:10};
      \draw[dashed] \fonctionL{0:10};

      % lignes verticales
      \draw[color=lightgray,dotted]
        (0,0) -- (A)
        (2.5,0) -- (B)
        (6,0) -- (C)
        (10,0) -- (10,10);
      \draw[dashed]
        (A) -- (0,10)
        (B) -- (2.5,10)
        (C) -- (6,10);

      % D1, D2, D3
      \draw[line width=2pt] (0,10) -- (10,10);
      \draw[(-)] (0,10) -- node[above]{$D_1$} (2.5,10);
      \draw[(-)] (2.5,10) -- node[above]{$D_2$} (6,10);
      \draw[(-)] (6,10) -- node[above]{$D_3$} (10,10);

      % socles de D1, D2, D3
      \draw[(-)] (0,0) -- node[above]{$\widehat{D}_1$} (2.5,0);
      \draw[(-)] (2.5,0) -- node[above]{$\widehat{D}_2$} (6,0);
      \draw[(-)] (6,0) -- node[above]{$\widehat{D}_3$} (10,0);

      % f
      \draw[dotted] (0,7.4) .. controls +(5,-.5) and +(-6,-2.8) .. (10,10);

      % \'Etiquettes
      \draw
        (2,6) node {$U$}
        (1.25,8.9) node {$C_{D_1}$}
        (4.25,9) node {$C_{D_2}$}
        (6.9,9.45) node {$C_{D_3}$}
        (5,5.2) node {$\mu$}
        (5,7.1) node {$f$}
        (3.65,7.7) node {$\lambda$};
      
% \foreach \k in {1,...,9}
%   {\draw[dotted,color=red] 
%     (0,\k) -- (10,\k)
%     (\k,0) -- (\k,10);
%   }

    \end{tikzpicture}
  \end{center}
  \caption{Dividing $A$ when $\nu=+\infty$.\label{fi:div-1}}
\end{figure}

Let us check the details now. Proposition~\ref{pr:maj-Z-aff} gives
a largely continuous affine function $\lambda:X\to\cZ$ such that
$\lambda(x)\geq\max(f(x,+\infty),\mu(x)+N_m)$ on $X$. Let $U$ be the set of $a\in
F_I(\Gamma^m)$ such that $\widehat{a}\in X$, $\mu(\widehat{a})\leq a_m\leq
\lambda(\widehat{a})$ and $a_m\equiv\rho\;[N_m]$. It is clearly a largely continuous
precell mod $N$ (with socle $X$ since $\lambda\geq\mu+N_m$). For each $D\in\cD$ let
$C_D$ be the set of $a\in F_I(\Gamma^m)$ such that $\widehat{a}\in\widehat{D}$,
$\lambda(\widehat{a})+1\leq a_m$ and $a_m\equiv\rho\;[N_m]$. This is a largely
continuous precell mod $N$ with socle $\widehat{D}$. Let $\cC=\{C_D\tq
D\in\cD\}$ and $\cU=\{U\}$. Obviously $\bigcup\cC=A\setminus U$, $\partial U=\partial B$ and {\modif (Fac),
(Diff) hold for every $D\in\cD$}.

By construction, for every $D\in\cD$ and every $a\in C_D$ we have
\begin{displaymath}
  \Delta_J(a)=a_m> \lambda(\widehat{a})\geq f(\widehat{a},+\infty)=f\big(\pi_J(a)\big)
\end{displaymath}
{\modif which proves (Sub')}. Further, for every $a\in A$ such that $\pi_J(a)\in D$
and $\Delta_J(a)\geq \lambda(\widehat{a})$, we have $\widehat{a}\in\widehat{D}$,
$a_m\geq\lambda(\widehat{a})$ and $a_m\equiv\rho\;[N_m]$ hence $a\in C_D$. {\modif This is
property (Sup)} with $\delta(b)=\lambda(\widehat{b})$ on $B$. 

\paragraph{Case 2:} $Y$ is a facet of $X$ and $B=Y\times\{+\infty\}$.
\\

Then $J=\widehat{J}$, $\bar\mu=+\infty$ on $Y$ (otherwise by
Proposition~\ref{pr:pres-face}, $F_{J\cup\{m\}}(A)\neq\emptyset$ is a proper face of
$A$ larger than $B$) and $\nu<+\infty$ (otherwise $X\times\{+\infty\}$ is a proper face
of $A$ larger than $B$). In particular $\mu(x)\geq f(y,+\infty)$ for every $y\in
Y$ and every $x\in X$ close enough to $y$, so there is a definable map
$\eta:Y\to\cZ$ such that for every $x\in X$
\begin{equation}
  \Delta_{\widehat{J}}(x)\geq \eta\big(\pi_{\widehat{J}}(x)\big) 
  \Rightarrow \mu(x)\geq f\big(\pi_{\widehat{J}}(x),+\infty\big). 
  \label{eq:le-fa-eta2}
\end{equation}

In the precells $C_D$ that we are looking for, we want to have
$a_m\geq\mu(\widehat{a})\geq\pi_{\widehat{J}}(\widehat{a})$ in order to get
condition (Sub). The idea is then to inflate first $\widehat{\cD}$ in
a way controlled by $\eta$ (using the induction hypothesis), and then to
divide $A$ by lifting this division of its socle $X$ (see
figure~\ref{fi:div-2}).

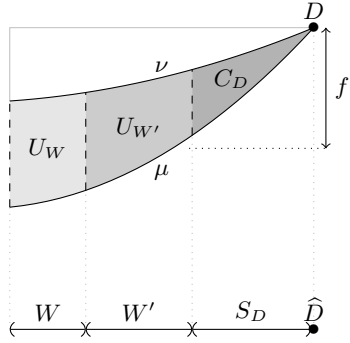
\begin{figure}[h]
  \small
  \begin{center}
    \begin{tikzpicture}[scale=.4]

      \def\fonctionNU#1{plot[domain=#1] (\x,{2.5*((\x+2)/12)^2+7.5})}
      \def\fonctionMU#1{plot[domain=#1] (\x,{6*((\x+1)/11)^2+4})}
      \newcommand{\cellule}[3]{
          \fill[color=gray!#1] 
            plot[domain=#2:#3] (\x,{2.5*((\x+2)/12)^2+7.5}) --
            plot[domain=#3:#2] (\x,{6*((\x+1)/11)^2+4}) -- cycle;
        }

      \coordinate (AN) at (0,7.5694444444444);
      \coordinate (BN) at (2.5,7.8515625000000);
      \coordinate (CN) at (6,8.61111111111111);
      \coordinate (AM) at (0,4.0495867768595);
      \coordinate (BM) at (2.5,4.6074380165289);
      \coordinate (CM) at (6,6.4297520661157);

      % Cellules U, U', CD
      \cellule{20}{0}{2.5};
      \cellule{40}{2.5}{6};
      \cellule{60}{6}{10};

      % lignes verticales et haut du cadre
      \draw[color=lightgray,dotted]
        (0,0) -- (AM)
        (2.5,0) -- (BM)
        (6,0) -- (CM)
        (10,0) -- (10,10);
      \draw[dashed]
        (AM) -- (AN)
        (BM) -- (BN)
        (CM) -- (CN);
      \draw[color=lightgray,very thin]
        (AN) -- (0,10) -- (10,10);

      \draw \fonctionMU{0:10};
      \draw \fonctionNU{0:10};
      \draw (10,10) node {$\bullet$} node[above] {$D$};
      \draw (10,0) node {$\bullet$} node[above] {$\widehat{D}$};

      % socles de W, W', SD
      \draw[(-)] (0,0) -- node[above]{$W$} (2.5,0);
      \draw[(-)] (2.5,0) -- node[above]{$W'$} (6,0);
      \draw[(-)] (6,0) -- node[above]{$S_D$} (9.9,0);

      % f
      \draw[<->] (10.4,10) -- node[right] {$f$} (10.4,6);
      \draw[dotted] (5.9,6) -- (10.5,6);

      % \'Etiquettes
      \draw
        (1.25,6) node {$U_W$}
        (4.25,6.7) node {$U_{W'}$}
        (7.3,8.3) node {$C_D$}
        (5,5.2) node {$\mu$}
        (5,8.7) node {$\nu$};
      
% \foreach \k in {1,...,9}
%   {\draw[dotted,color=red] 
%     (0,\k) -- (10,\k)
%     (\k,0) -- (\k,10);
%   }

    \end{tikzpicture}
  \end{center}
  \caption{Dividing $A$ when $B=Y\times\{+\infty\}$ and $Y$ is a facet of $X$.\label{fi:div-2}}
\end{figure}

The induction hypothesis applies to $X$, $Y$, $\widehat{\cD}$ and
$g(y)=\max(f(y,+\infty),\eta(y))$ on $Y$. It gives a definable map $\varepsilon:Y\to\cZ$
and a pair $(\cS,\cW)$ of families of precells. For each $W\in\cW$ (resp.
$D\in\cD$) let $U_W$ (resp. $C_D$) be the set of $a\in F_J(\Gamma^m)$ such that
$\widehat{a}\in W$ (resp. $\widehat{a}$ belongs to the unique precell
$S_{\widehat{D}}\in\cS$ whose facet is $\widehat{D}$), $\mu(\widehat{a})\leq
a_m\leq\nu(\widehat{a})$ and $a_m\equiv\rho\;[N_m]$. This is obviously a largely
continuous precell mod $N$ with socle $W$ (resp. $S_{\widehat{D}}$), and
exactly the set of $a\in A$ such that $\widehat{a}\in W$) (resp.
$S_{\widehat{D}}$). In particular it is contained in $A$, and if we
let $\cU=\{U_W\tq W\in\cW\}$ and $\cC=\{C_D\tq D\in\cD\}$ then $\cU$ is a
partition $A\setminus\bigcup\cC$ by induction hypothesis on $(\cS,\cW)$. 

For every $W\in\cW$, every proper face of $W$ is a proper face $Z$ of
$X$. Let $H$ be its support. Then by Proposition~\ref{pr:pres-face},
$(\bar\mu_{|Z},\bar\nu_{|Z},\rho)$ is a presentation of $F_H(U_W)$, but
also of $F_H(A)$ hence $F_H(U_W)=F_H(A)$ is a proper face of $A$. 

{\modif Let us check (Fac') and (Diff).} For
every $D\in\cD$, since $\bar\mu=+\infty$ on $Y$ we have
$F_J(C_D)=\widehat{D}\times\{+\infty\}=D$  by Proposition~\ref{pr:pres-face},
hence $\pi_J(C_D)=D$ by Proposition~\ref{pr:face-egale-proj}.
Moreover for every $E\in\cE$, $\pi_{\widehat{J}}(S_{\widehat{D}}\setminus
S_{\widehat{E}}) \subseteq \widehat{D}\setminus\widehat{E}$ by induction hypothesis
hence
\begin{displaymath}
  \pi_J(C_D\setminus C_E)
  = \big[\pi_{\widehat{J}}(S_{\widehat{D}})
           \setminus \pi_{\widehat{J}}(S_{\widehat{E}})\big] \times \{+\infty\}
  \subseteq (\widehat{D}\setminus\widehat{E}) \times \{+\infty\}
  = D\setminus E.
\end{displaymath}

{\modif Now we turn to (Sub').} For every $a\in C_D$, since $J=\widehat{J}$ we
have $\Delta_J(a)=\min(a_m,\Delta_{\widehat{J}}(\widehat{a}))$ and
$\pi_J(a)=(\pi_{\widehat{J}}(\widehat{a}),+\infty)$. By the induction hypothesis
$\Delta_{\widehat{J}}(\widehat{a})\geq\eta\circ\pi_{\widehat{J}}(\widehat{a})$ and
$\Delta_{\widehat{J}}(\widehat{a})\geq f(\pi_{\widehat{J}}(\widehat{a}),+\infty)$
because $\widehat{a}\in S_{\widehat{D}}$. The first inequality implies
that $a_m\geq\mu(x)\geq f(\pi_{\widehat{J}}(\widehat{a}),+\infty)$ by
(\ref{eq:le-fa-eta2}). Together with the second inequality this gives
that $\min(a_m,\Delta_{\widehat{J}}(\widehat{a}))\geq
f(\pi_{\widehat{J}}(\widehat{a}),+\infty)$. That is $\Delta_J(a)\geq f(\pi_J(a))$.

We finally check (Sup) {\modif with $\delta(b)=\varepsilon(\widehat{b})$ on
$B$}. Since $C_D$ is clearly the set of $a\in A$ such that
$\widehat{a}\in S_{\widehat{D}}$, for every $a\in A$ such that $\pi_J(a)\in D$
(hence $\pi_{\widehat{J}}(\widehat{a})\in\widehat{D}$) and
$\Delta_J(a)\geq\varepsilon\circ\pi_{\widehat{J}}(\widehat{a})$ we have $\widehat{a}\in
S_{\widehat{D}}$ by induction hypothesis on $\varepsilon$ and $\widehat{D}$
hence $a\in C_D$.

\paragraph{Case 3:} $Y$ is a facet of $X$ and $B=(Y\times\cZ)\cap\overline{A}$.
\\

Then $m\in\Supp B = J$, hence
$\bar\mu<+\infty$ on $Y$, $\mu_D<+\infty$ for every $D\in\cD$, and for every $a\in A$:
\begin{equation}
  \Delta_J(a)=\Delta_{\widehat{J}}(\widehat{a})
  \quad\mbox{and}\quad 
  \pi_J(a)=(\pi_{\widehat{J}}(\widehat{a}),a_m)
  \label{eq:delta-J}
\end{equation}
Note that $\rho=\rho_D$ for every $D\in\cD$ because, given any $b\in
D\subseteq B$, we have $b_m\neq+\infty$ and on one hand $b_m\equiv\rho_D\;[N_m]$, on the
other hand $b_m\equiv\rho\;[N_m]$ (using the presentation of $B=F_J(A)$
given by Proposition~\ref{pr:pres-face}). 

\paragraph{\em Sub-case 3.1:} $\nu<\infty$. 
\\

{\modif
  This is the most difficult case, because $f(b)$ depends both on
$\widehat{b}$ and $b_m$. Therefore our construction is done in two
steps. 
\\

{\em Step 1}:
Intuitively, we are going to remove the top of $A$ by introducing a
function $\zeta(x)$ which will ensure that $a_m$ doesn't grow too fast as
$a\in A$ goes closer to $B$. The connection between $\zeta$ (a function of
$x\in\widehat{A}$) and $f$ (a function of $b\in B$) will be made {\it via}
an intermediate function $g$ defined below. We need to restrict the
socle $X$ of $A$ to a domain $X^\circ$ close enough to $Y$ so as to ensure
at least that $\mu<\zeta<\nu$ on $X^\circ$. In order to do this we will divide $X$ by
applying to it the induction hypothesis. The resulting partition of
$X$ together with $\zeta$ will give us a partition of $A$ which might look
like figure~\ref{fi:div-3-top}.
}

\begin{figure}[h]
  \small
  \begin{center}
    \begin{tikzpicture}[scale=.4]
      \draw[thin,color=gray!50] (0,0) rectangle (10,10);
   
      \def\fonctionNU#1{plot[domain=#1] (\x,{6.5*((\x+3)/13)^3+3.5})}
      \def\courbeF{(8,2.5) .. controls +(0,7) and +(-1,-5) .. (10,10)} 
      \def\courbeZ{(7,5.5) .. controls +(3,1) and +(-.1,-1) .. (10,10)}
   
      % A0
      \fill[color=gray!60] \courbeZ -- (10,2.5) -- (7,2.5) --cycle;
      % V
      \fill[color=gray!20] \fonctionNU{7:10} -- \courbeZ -- (7,5.5) --cycle;
      % U_{W_2}
      \fill[color=gray!40] \fonctionNU{4:7} -- (7,2.5) -- (4,2.5) --cycle;
      % U_{W_1}
      \fill[color=gray!20] \fonctionNU{0:4} -- (4,2.5) -- (0,2.5) --cycle; 
   
      % B
      \draw[line width=2pt] (10.03,2.5) -- node[right]{$B$} (10.03,10);
   
      % Pointill\'es de projection verticale
      \draw[dotted,color=lightgray]
        (0,0) -- (0,2.5)
        (4,0) -- (4,2.5)
        (7,0) -- (7,2.5);
   
      % \nu,\mu
      \draw \fonctionNU{0:10};
      \draw plot[domain=0:10] (\x,2.5);
   
      % \zeta
      \draw[dashed] \courbeZ 
        (4,2.5) -- (4,4.5147928994083)
        (7,2.5) -- (7,6.4585798816568)
        (0,2.5) -- (0,3.5798816568047);
   
      % f
      \draw[dotted,thick] \courbeF ;
   
      % W1, W2, X0
      \draw[(-)] (0,0) -- node[above]{$W_1$} (4,0);
      \draw[(-)] (4,0) -- node[above]{$W_2$} (7,0);
      \draw[(-)] (7,0) -- node[above]{$X^\circ$} (10,0);
   
      % \'Etiquettes 
      \draw (7.6,6.4) node{$V$};
      \draw (9,3.4) node{$A^\circ$};
      \draw (5.5,3.15) node{$U_{W_2}$};
      \draw (2,3.15) node{$U_{W_1}$};
      \draw (8,4.6) node[right] {$f$};
      \draw (9.2,6.5) node{$\zeta$};
      \draw (5.5,2.1) node{$\mu$};
      \draw (5.5, 5.8) node{$\nu$};

    \end{tikzpicture} 

  \end{center}
  \caption{Removing the top of $A$.\label{fi:div-3-top}}
\end{figure}
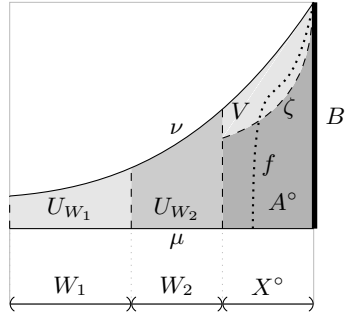  

Let $g:Y\to\cZ$ be a positive affine map given by
Proposition~\ref{pr:maj-Z-aff} such that $g(y)\geq f(y,0)+\alpha(\bar\mu(y)+
N_m)$ on $Y$ and $\bar g=+\infty$ on $\partial Y$. Given any $y\in Y$, since
$g(y)<+\infty$ and $\nu-\mu$ has limit $+\infty$ at $y$, we have
$\nu(x)-\mu(x)>2N_m+1+g(y)$ for every $x\in X$ close enough to
$y$. So there is a definable function $\eta_1:Y\to\cZ$ such that
for every $x\in X$
\begin{equation}
  \Delta_{\widehat{J}}(x)\geq\eta_1\big(\pi_{\widehat{J}}(x)\big) \Rightarrow 
  \nu(x)-\mu(x)>2N_m+1+g\big(\pi_{\widehat{J}}(x)\big).
  \label{eq:le-fa-eta31}
\end{equation}
The induction hypothesis applies to $X$, $Y$, $\{Y\}$ and $\max(\eta_1,2g)$.
It gives a definable map $\varepsilon_1:Y\to\cZ$ and a pair $(\cS_1,\cW_1)$ of
families of precells. In the present case $\cS_1$ consists of a single
largely continuous precell $X^\circ$ mod $N$ contained in $X$, such that
$\Delta_{\widehat{J}}\geq\max(\eta_1\circ\pi_{\widehat{J}},2g\circ\pi_{\widehat{J}})$ on
$X^\circ$, and every $x\in X$ such that $\pi_{\widehat{J}}(x)\in Y$ and
$\Delta_{\widehat{J}}(x)\geq\varepsilon_1(\pi_{\widehat{J}}(x))$ belongs to $X^\circ$. The
family $\cW_1$ is a finite partition of $X\setminus X^\circ$ in largely continuous
precells mod $N$. Let $\cU_1=\{U_W\tq W\in\cW_1\}$ where $U_W=(W\times\cZ)\cap A$
for every $W\in\cW$. Since $\nu<+\infty$, the proper faces of $U_W$ are proper
faces of $A$ by Claim~\ref{cl:U-pas-B}. 

For every $k\notin\widehat{J}$ and every $x\in X^\circ$, we have
$x_k\geq\Delta_{\widehat{J}}(x)$ because $k\notin\widehat{J}$, and $\Delta_{\widehat{J}}
\geq 2g\circ\pi_{\widehat{J}}(x)$ on $X^\circ$ by the induction hypothesis. Thus on one
hand $x_k-g\circ\pi_{\widehat{J}}(x)\geq g\circ\pi_{\widehat{J}}(x)\geq1$, and on the
other hand $x_k-g\circ\pi_{\widehat{J}}(x)\geq x_k/2$. In particular $x\mapsto
x_k-g\circ\pi_{\widehat{J}}(x)$ is a largely continuous positive
affine function on $X^\circ$ with limit $+\infty$ at every point of $\partial X^\circ$.
We also have $\Delta_{\widehat{J}}(x)\geq\eta_1(\pi_{\widehat{J}}(x))$ by induction
hypothesis, hence $\nu(x)-\mu(x)>2N_m+1+g(\pi_{\widehat{J}}(x))$ by
(\ref{eq:le-fa-eta31}). In particular the restriction of $\nu-\mu-2N_m-1$
to $X^\circ$ is a positive affine function with limit $+\infty$ at
every point of $\partial X^\circ$. Proposition~\ref{pr:min-Z-aff} then gives a largely
continuous positive affine function $\lambda:X^\circ\to\cQ$ such that
$\bar\lambda=+\infty$ on $\partial X^\circ$, $\lambda\leq\nu-\mu-2N_m-1$ on $X^\circ$ and
$\lambda(x)\leq(x_k-g\circ\pi_{\widehat{J}}(x))/\alpha$ for every $k\notin\widehat{J}$. 
Let us quote for further use that in particular
\begin{equation}
  \alpha\lambda(x) \leq \min_{k\notin\widehat{J}}\big(x_k-g(\pi_{\widehat{J}}(x))\big)
        = \Delta_{\widehat{J}}(x)-g(\pi_{\widehat{J}}(x)).
  \label{eq:le-fa-lambda}
\end{equation}
Note that $\partial X^\circ=\overline{Y}$ because $X^\circ$ has a unique facet which
is $Y$ by Claim~\ref{cl:fac}, hence $\bar\lambda=+\infty$ on $\overline{Y}$.
Let $n\geq1$ an integer such that $n\lambda$ is integrally affine, so that
$\lambda(x)>t$ if and only if $\lambda(x)\geq t+1/n$ for every $(x,t)\in X^\circ\times\cZ$. 
Let $\zeta=\mu+\lambda+N_m$ on $X^\circ$, and $V$ (resp. $A^\circ$) be the set
of $a\in F_I(\Gamma^m)$ such that $\widehat{a}\in X^\circ$, $\zeta(\widehat{a})+1/n\leq
a_m\leq \nu(\widehat{a})$ (resp. $\mu(\widehat{a})\leq a_m\leq\zeta(\widehat{a})$) and
$a_m\equiv\rho\;[N_m]$. By construction $\zeta$ is a largely continuous affine
map on $X^\circ$ with $\bar\zeta=+\infty$ on $\partial X^\circ$. Moreover on $X^\circ$ we have
\begin{displaymath}
  \zeta+\frac{1}{n}+N_m=\mu+\lambda+2N_m+\frac{1}{n}\leq \nu
\end{displaymath}
(because $\lambda\leq\nu-\mu-2N_m-1$ by construction) hence the socle of
$V$ is $X^\circ$. Obviously $\mu+N_m\leq\mu+\lambda+N_m=\zeta$ (because $\lambda>0$ by
construction) hence the socle of $A^\circ$ is $X^\circ$. Thus both $V$ and
$A^\circ$ are largely continuous precells mod $N$ contained in $(X^\circ\times\cZ)\cap
A$. Moreover $a_m>\zeta(\widehat{a})=\mu(\widehat{a})+\lambda(\widehat{a})+N_m$
if and only if $a_m \geq \mu(\widehat{a})+\lambda(\widehat{a})+1/n+N_m =
\zeta(\widehat{a})+1/n$. Thus $V$ and $A^\circ$ form a partition of
$(X^\circ\times\cZ)\cap A$, or equivalently $\cU_1\cup\{V\}$ is a partition of 
$A\setminus A^\circ$. Since $\bar\zeta=+\infty$ on $\partial X^\circ=\overline{Y}$, by
Proposition~\ref{pr:pres-face} every proper face $V'$ of $V$ is of type
$Z\times\{+\infty\}$ for $Z$ a face of $Y$. In particular $V'$ is a proper face of
$A$. 
\\

{\modif
{\em Step 2}:
Intuitively, we are going to build the $C_D$'s by inflating inside
$A^\circ$ each $D$ in $\cD$ as suggested by figure~\ref{fi:div-3-right}
(which zooms in on $A^\circ$, the other parts of $A$ remaining
as in figure~\ref{fi:div-3-top}).
}

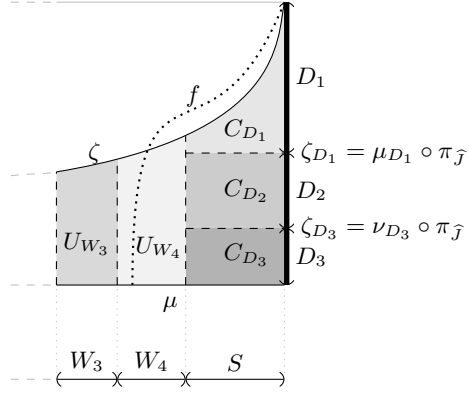
\begin{figure}[h]
  \small
  \begin{center}

    \begin{tikzpicture}[yscale=.5]

      \def\courbeF{(8,2.5) .. controls +(0,7) and +(-1,-5) .. (10,10)} 
      \def\courbeZ{(7,5.5) .. controls +(3,1) and +(-.1,-1) .. (10,10)}
   
      % C_D1
      \begin{scope}
        \clip (8.7,6) rectangle (10,10);
        \fill[color=gray!20] (7,2.5) -- \courbeZ -- (10,2.5) --cycle;
      \end{scope}

      % C_D2
      \fill[color=gray!40] (8.7,4) rectangle (10,6);

      % C_D3
      \fill[color=gray!60] (8.7,2.5) rectangle (10,4);

      % U_W4
      \begin{scope}
        \clip (7.8,0) rectangle (8.7,10);
        \fill[color=gray!10] (7,2.5) -- \courbeZ -- (10,2.5) --cycle;
      \end{scope}

      % U_W3
      \begin{scope}
        \clip (7,0) rectangle (7.8,10);
        \fill[color=gray!30] (7,2.5) -- \courbeZ -- (10,2.5) --cycle;
      \end{scope}

     % Pointill\'es de projection verticale
     \draw[dotted,color=lightgray]
       (7,0) -- (7,2.5)
       (7.8,0) -- (7.8,2.5)
       (8.7,0) -- (8.7,2.5)
       (10,0) -- (10,2.5);

      \draw[dashed]
        (7,2.5) -- (7,5.5)
        (7.8,2.5) -- (7.8,5.82)
        (8.7,2.5) -- (8.7,6.45)
        (8.7,4) -- (10,4) ++(0.1,0) node[right] {$\zeta_{D_3}=\nu_{D_3}\circ\pi_{\widehat{J}}$}
        (8.7,6) -- (10,6) ++(0.1,0) node[right] {$\zeta_{D_1}=\mu_{D_1}\circ\pi_{\widehat{J}}$};

      % Haut du cadre et prolongement \`a gauche
      \draw[color=lightgray,thin] (7,10) -- (10,10);
      \draw[color=lightgray,dashed] (6.4,10) -- (7,10);
      \draw[color=lightgray,dashed] (6.4,5.4) .. controls +(.5,.1) and +(-.3,-.1) .. (7,5.5);
      \draw[color=lightgray,dashed] (6.4,2.5) -- (7,2.5);
      \draw[color=lightgray,dashed] (6.4,0) -- (7,0);

      % B
      \draw[line width=2pt] (10.03,2.5) -- (10.03,10);

      % D1, D2, D3
      \draw[(-)] (10.03,2.5) -- node[right] {$D_3$} (10.03,4);
      \draw[(-)] (10.03,4) -- node[right] {$D_2$} (10.03,6);
      \draw[(-)] (10.03,6) -- node[right] {$D_1$} (10.03,10);

      % \zeta,\mu,f
      \draw \courbeZ;
      \draw plot[domain=7:10] (\x,2.5);
      \draw[dotted,thick] \courbeF;

     % W3, W4, S
     \draw[(-)] (7,0) -- node[above]{$W_3$} (7.8,0);
     \draw[(-)] (7.8,0) -- node[above]{$W_4$} (8.7,0);
     \draw[(-)] (8.7,0) -- node[above]{$S$} (10,0);

     % \'Etiquettes
     \draw 
       (7.4,3.6) node{$U_{W_3}$}
       (8.35,3.5) node{$U_{W_4}$}
       (9.5,3.3) node{$C_{D_3}$}
       (9.5,5) node{$C_{D_2}$}
       (9.5,6.6) node{$C_{D_1}$}
       (8.8,7.5) node{$f$}
       (7.5,6) node{$\zeta$}
       (8.5,2) node{$\mu$};

    \end{tikzpicture}
  \end{center}
  \caption{Dividing $A^\circ$ by inflating each $D\in\cD$.\label{fi:div-3-right}}
\end{figure}

For every $D\in\cD$ let $\zeta_D=\nu_D$ if $\nu_D<+\infty$ and $\zeta_D=\mu_D+N_m$
otherwise. Since $\bar\zeta=+\infty$ on $Y$ there is a definable function
$\eta_2:Y\to\cZ$ such that for every $x\in X^\circ$ and every $D\in\cD$ such that
$\pi_{\widehat{J}}(x)\in\widehat{D}$ we have
\begin{equation}
  \Delta_{\widehat{J}}(x)\geq \eta_2\big(\pi_{\widehat{J}}(x)\big) \Rightarrow
  \zeta(x)\geq \zeta_D\big(\pi_{\widehat{J}}(x)\big).
  \label{eq:le-fa-eta32}
\end{equation}

The induction hypothesis applies to $X^\circ$, $Y$, $\widehat{\cD}$ and
$\eta_2$. It gives a definable map $\varepsilon_2:Y\to\cZ$ and a pair $(\cS_2,\cW_2)$
of families of precells. For each $W\in\cW_2$ let $U_W=(W\times\cZ)\cap A^\circ$.
Clearly the family $\cU_2=\{U_W\tq W\in\cW_2\}$ is a finite partition in
largely continuous precells mod $N$ of the complement in $A^\circ$ of the set
$A^{\circ\circ}=(\bigcup\cS_2\times\cZ)\cap A^\circ$. Equivalently, $\cU_1\cup\{V\}\cup\cU_2$ is a
finite partition of $A\setminus A^{\circ\circ}$. Since $\nu<+\infty$, by
Claim~\ref{cl:U-pas-B} the proper faces of $U_W$ are proper faces of
$A$ for every $W\in\cW_2$. 

For each $D\in\cD$ let $S_{\widehat{D}}$ be the precell in $\cS_2$ given by
the induction hypothesis, so that conditions (Fac), (Sub), (Sup), (diff)
apply to $S_{\widehat{D}}$, $\eta_2$ and $\varepsilon_2$. If $\nu_D=+\infty$ (resp.
$\nu_D<+\infty$) let $C_D$ be the set of $a\in F_I(\Gamma^m)$ such that
$\widehat{a}\in S_{\widehat{D}}$, $\mu_D(\pi_J(\widehat{a}))\leq a_m\leq
\zeta(\widehat{a})$ (resp. $\mu_D(\pi_J(\widehat{a}))\leq a_m\leq
\nu_D(\pi_J(\widehat{a}))$) and $a_m\equiv\rho\;[N_m]$.

{\modif Let us check that $C_D$ is a largely continuous precell mod
$N$.} For every $x\in S_{\widehat{D}}$ we have
$\pi_{\widehat{J}}(x)\in\widehat{D}$, because
$\widehat{D}=F_{\widehat{J}}(S_{\widehat{D}})$ by (Fac), and
$F_{\widehat{J}}(S_{\widehat{D}})=\pi_{\widehat{J}}(S_{\widehat{D}})$ by
Proposition~\ref{pr:face-egale-proj}(\ref{it:face-proj}). So there is
$b\in D$ such that $\widehat{b}=x$, $\mu_D(\pi_{\widehat{J}}(x))\leq
b_m\leq\nu_D(\pi_{\widehat{J}}(x))$ and $b_m\equiv\rho\;[N_m]$.  We can (and do)
require in addition that $b_m\leq\mu_D(\pi_{\widehat{J}}(x))+N_m$, hence
$b_m\leq\zeta_D(\pi_{\widehat{J}}(x))$. Because $x\in S_{\widehat{D}}$ we also
have $\Delta_{\widehat{J}}(x)\geq\eta_2\circ\pi_{\widehat{J}}(x)$ by (Sub), hence
$\zeta_D(\pi_{\widehat{J}}x)\leq\zeta(x)$ by (\ref{eq:le-fa-eta32}). Altogether
this proves that $(x,b_m)\in C_D$, hence $x$ belongs to the socle of
$C_D$. So the socle of $C_D$ is exactly $S_{\widehat{D}}$ and $C_D$ is
then a largely continuous precell mod $N$. 

{\modif Now we turn to (Fac').} The presentation of
$F_J(C_D)$ given by Proposition~\ref{pr:pres-face} is exactly
$(\mu_D,\nu_D,\rho)$, hence $F_J(C_D)=D$ since $\rho_D=\rho$. In particular
$\pi_J(C_D)=D$ by
Proposition~\ref{pr:face-egale-proj}(\ref{it:face-proj}). More
precisely, the above computations show that we have
\begin{equation}
  C_D=\big\{a\in A^{\circ\circ}\tq \widehat{a}\in S_{\widehat{D}} \mbox{ and } 
                         \pi_J(a)\in D\big\}.
  \label{eq:le-fa-CD}
\end{equation}

Let $\cC=\{C_D\tq D\in\cD\}$ and $\cU=\cU_1\cup\{V\}\cup\cU_2$. We already know
that $\cU$ is a finite partition of $A\setminus A^{\circ\circ}$ in largely continuous
precells mod $N$ whose proper faces are proper faces of $A$, and that
each $C_D\in\cC$ is a largely continuous precell mod $N$ contained in
$A^{\circ\circ}$ with socle $S_{\widehat{D}}$ and $F_J(C_D)=\pi_J(C_D)=D$.
Let us check that $\bigcup\cC=A^{\circ\circ}$. In order to do so, we are claiming
that
\begin{equation}
  \forall a\in A^{\circ\circ},\forall E\in\cD,\ \pi_J(a)\in E \Rightarrow a\in C_E.
  \label{eq:le-fa-claim}
\end{equation}
Assume the contrary and let $a\in A^{\circ\circ}$, $E\in\cE$ be such that $\pi_J(a)\in E$
and $a\notin C_E$. By (\ref{eq:le-fa-CD}) this implies that $\widehat{a}\notin 
S_{\widehat{E}}$. But the socle of $A^{\circ\circ}$ is $\bigcup\cS_2$, hence
$\widehat{a}\in S_{\widehat{D}}$ for some $D\in\cD$. Thus
$\pi_{\widehat{J}}(\widehat{a})$ belongs to
$\pi_{\widehat{J}}(S_{\widehat{D}}\setminus S_{\widehat{E}})$. By the induction
hypothesis the latter is contained in $\widehat{D}\setminus\widehat{E}$, hence
$\pi_{\widehat{J}}(\widehat{a})\notin\widehat{E}$. But
$\pi_{\widehat{J}}(\widehat{a})$ is also the socle of $\pi_J(a)$. Since
$\pi_J(a)\in E$ it follows that $\pi_{\widehat{J}}(\widehat{a})\in\widehat{E}$, a
contradiction.

That $A^{\circ\circ}\subseteq\bigcup\cC$ then follows immediately from
(\ref{eq:le-fa-claim}) and the fact that $\pi_J(A^{\circ\circ})\subseteq\pi_J(A)=B\subseteq\bigcup\cD$.
So $A^{\circ\circ}=\bigcup\cC$ and it only remains to check (Sub'), (Sup) and
(Diff') for any fixed $D\in\cD$.

We start with (Diff'). Pick any $E\in\cD$, assume that there is a point
$b$ in $\pi_J(C_D\setminus C_E)$ which belongs to $E$. Then $b=\pi_J(a)$ for some
$a\in C_D\setminus C_E$. We have $a\in A^{\circ\circ}$ and $a\notin C_E$, hence $\pi_J(a)\notin E$ by
(\ref{eq:le-fa-claim}), that is $b\notin E$, a contradiction. Hence
$\pi_J(C_D\setminus C_E)$ is disjoint from $E$. 

Let us now turn to (Sup). For every $b\in B$, since $\bar\zeta=+\infty$ on $\partial
X^\circ=\overline{Y}$ and $\widehat{b}\in Y$, we have $\zeta(x)\geq b_m$ whenever $x\in
X^\circ$ is close enough to $\widehat{b}$ (that is whenever
$\pi_{\widehat{J}}(x)=\widehat{b}$ and $\Delta_{\widehat{J}}(x)$ is large
enough). So there is a definable function $\eta_3:B\to\cZ$ such that for
every $a\in (X^\circ\times\cZ)\cap A$
\begin{equation}
  \Delta_{\widehat{J}}(\widehat{a})\geq\eta_3\big(\pi_J(a)\big) \Rightarrow 
  \zeta(\widehat{a})\geq a_m.
  \label{eq:le-fa-eta33}
\end{equation}
Let $\delta:b\in B\mapsto\max(\varepsilon_1(\widehat{b}),\eta_3(b),\varepsilon_2(\widehat{b}))$. For
every $a\in A$ such that $\pi_J(a)\in D$ and $\Delta_J(a)\geq\delta\circ\pi_J(a)$, since
$\Delta_J(a)=\Delta_{\widehat{J}}(\widehat{a})$ by (\ref{eq:delta-J}) we have in
particular $\pi_{\widehat{J}}(\widehat{a})\in Y$ and
$\Delta_{\widehat{J}}(\widehat{a})\geq\varepsilon_1(\pi_{\widehat{J}}(\widehat{a}))$, 
hence $\widehat{a}\in X^\circ$ by construction. So $a\in (X^\circ\times\cZ)\cap A$ and
$\Delta_{\widehat{J}}(\widehat{a})\geq\eta_3\big(\pi_J(a)\big)$, which implies that
$a_m\leq \zeta(\widehat{a})$ by (\ref{eq:le-fa-eta33}), hence $a\in A^\circ$ by
construction. On the other hand, since $\widehat{a}\in X^\circ$,
$\pi_{\widehat{J}}(\widehat{a})\in\widehat{D}$ and
$\Delta_{\widehat{J}}(\widehat{a})\geq\varepsilon_2(\pi_{\widehat{J}}(\widehat{a}))$, we
get that $\widehat{a}\in S_{\widehat{D}}$ by construction. In particular
$\widehat{a}\in\bigcup\cS_2$, hence $a\in A^{\circ\circ}$ since $A^{\circ\circ}=(\bigcup\cS_2\times\cZ)\cap
A^\circ$. Altogether we have $a\in A^{\circ\circ}$, $\widehat{a}\in S_{\widehat{D}}$
and $\pi_J(a)\in D$ hence that $a\in C_D$ by (\ref{eq:le-fa-CD}), which
proves (Sup). 

It only remains to check (Sub'), that is $\Delta_J\geq f\circ\pi_J$ on $C_D$. 
This is the moment to recall (\ref{eq:le-fa-lambda}), which says that
$\alpha\lambda\leq\Delta_{\widehat{J}}-g\circ\pi_{\widehat{J}}$ on $X^\circ$. Recall also that
$g(y)\geq f(y,0)+\alpha(\bar\mu(y)+N_m)$ on $Y$ by definition of $g$. Thus on
$X^\circ$ we have
\begin{equation}
  \alpha\lambda(x)
  \leq \Delta_{\widehat{J}}(x)
    - f\big(\pi_{\widehat{J}}(x),0\big) 
    - \alpha \bar\mu\big(\pi_{\widehat{J}}(x)\big) 
    - \alpha N_m. 
  \label{eq:le-fa-fin-lambda}
\end{equation}
For every $a\in C_D$, $\widehat{a}\in X^\circ$ and $a\in A^\circ$ hence
$a_m\leq\zeta(\widehat{a})=\mu(\widehat{a})+\lambda(\widehat{a})+N_m$. We also have
$\mu(\widehat{a})=\bar\mu\big(\pi_{\widehat{J}}(\widehat{a})\big)$ by
Proposition~\ref{pr:prol-pi}. Combining all this with
(\ref{eq:le-fa-fin-lambda}) we get that
\begin{equation}
  \alpha a_m \leq {\modif \alpha\bar\mu\big(\pi_{\widehat{J}}(x)\big) + \alpha \lambda(\widehat{a}) +\alpha N_m}
        \leq \Delta_{\widehat{J}}(\widehat{a})
          - f\big(\pi_{\widehat{J}}(\widehat{a}),0\big).
  \label{eq:le-fa-fin-am}
\end{equation}
Since $f(\pi_J(a))=f\big(\pi_{\widehat{J}}(\widehat{a}),0\big) +\alpha a_m$ by
the definition of $f$, and $\Delta_J(a)=\Delta_{\widehat{J}}(\widehat{a})$ by
(\ref{eq:delta-J}), we finally get from (\ref{eq:le-fa-fin-am}) that
$\Delta_J(a)=\Delta_{\widehat{J}}(\widehat{a})\geq f(\pi_J(a))$.

\paragraph{{\em Sub-case 3.2:}} $\nu=+\infty$. 
\\

This final case is easy: we simply divide $A$ in two pieces, above and
below a function $\lambda$ to be defined, so that the previous sub-case~3.2
applies to the lower part of $A$. The upper part $A$ doesn't require
any special treatment: it will simply be incorporated in the family
$\cU$. 

Let us check the details now. 
Proposition~\ref{pr:maj-Z-aff} gives a largely continuous integrally
affine map $\lambda$ on $X$ such that $\bar\lambda=+\infty$ on $\partial X$ and $\lambda\geq \mu+N_m$.
Let $A^-$ (resp. $A^+$) be the set of $a\in F_I(\Gamma^m)$ such that
$\widehat{a}\in X$, $\mu(\widehat{a})\leq a_m\leq \lambda(\widehat{a})$ (resp.
$\lambda(\widehat{a})+1\leq a_m$) and $a_m\equiv\rho\;[N_m]$. Its socle is $X$ (for
$A^-$ we use that $\lambda\geq \mu+N_m$) hence it is a largely continuous precell
mod $N$. Since $\lambda$ takes values in $\cZ$, $A^-$ and $A^+$ form a
partition of $A$. The presentation of the faces of $A$, $A^-$ $A^+$
given by Proposition~\ref{pr:pres-face} gives that every proper face
of $A^-$ and $A^+$ is a proper face of $A$, and $B$ is a face of
$A^-$. The previous sub-case~3.1 applies to $A^-$, $B$, $\cD$ and $f$.
It gives a pair $(\cC^-,\cW^-)$ of families of largely continuous
precells mod $N$ and an integrally affine map $\delta^-:B\to\cZ$. Then
$(\cC^-,\cW^-\cup\{A^+\})$ and $\delta^-$ have all the required properties for
$A$, $\cD$ and $f$, except possibly (Sup). We remedy this by
replacing $\delta^-$ by a larger function $\delta$ defined as follows. 

For every $b\in B$, we have $\lambda(x)\geq b_m$ for every $x\in X$ close
enough to $\widehat{b}$ since $\bar\lambda=+\infty$ on $Y$. So there is a
definable function $\eta:B\to\cZ$ such that for every $a\in A$
\begin{equation}
  \Delta_J(a)\geq \eta\big(\pi_J(a)\big) \Rightarrow \lambda(\widehat{a})\geq a_m.
  \label{eq:le-fa-eta-final}
\end{equation}
Let $\delta=\max(\eta,\delta^-)$, then for every $D\in\cD$ and every $a\in A$ such that
$\pi_J(a)\in D$ and $\Delta_J(a)\geq \delta(\pi_J(a))$ we have in particular
$\Delta_J(a)\geq\eta(\pi_J(a))$ hence $a_m\leq \lambda(\widehat{a})$ by
(\ref{eq:le-fa-eta-final}), that is $a\in A^-$. Moreover we
have $\pi_J(a)\in D$ and $\Delta_J(a)\geq\delta^-(\pi_J(a))$. Altogether this implies
that $a$ belongs to $C_D\in\cC^-$, which in turn proves (Sup).
\end{proof}

\begin{theorem}[Monohedral Division]\label{th:mono-div}
  Let $A\subseteq F_I(\Gamma^m)$ be a largely continuous precell mod $N$, $f:\partial A\to\cZ$
  a definable function, and $\cD$ a complex of monohedral largely
  continuous precells mod $N$ such that $\bigcup\cD=\partial A$. Then there exists a
  finite partition $\cC$ of $A$ in monohedral largely continuous precells
  mod $N$ such that $\cC\cup\cD$ is a closed complex, $\cC$ contains for
  every $D\in\cD$ a unique precell $C$ with facet $D$, and moreover $\Delta_J\geq
  f\circ\pi_J$ on $C$ where $J=\Supp D$. 
\end{theorem}

\begin{proof}
  The proof goes by induction on the number $n$ of proper faces of
  $A$. If $n=0$ then $\cD=\emptyset$ and $A$ is monohedral, hence
  $\cC=\{A\}$ gives the conclusion. So let us assume that $n\geq1$ and the
  result is proved for smaller integers. Let $B$ be a facet of $A$.
  Lemma~\ref{le:face-elargie} applied to $A$, $B$, $\cD$ and the
  restriction of $f$ to $B$ gives a pair $(\cC_B,\cU)$ of families of
  precells. For every $U\in\cU$, the proper faces of $U$ are proper faces
  of $A$. So the family $\cD_U=\{D\in\cD\tq D\subseteq\partial U\}$ is a complex and
  $\bigcup\cD_U=\partial U$. Since $B$ is not a proper face of $U$  by
  Claim~\ref{cl:U-pas-B}, the induction hypothesis applies to $U$,
  $\cD_U$ and the restriction of $f$ to $\partial U$. It gives a family
  $\cC_U$ of precells. Let $\cC$ be the union of $\cC_B$ and $\cC_U$ for
  $U\in\cU$. This is a family of largely continuous precells mod $N$
  partitioning $A$. By construction $\cC$ contains for every $D\in\cD$ a
  unique precell $C$ with facet $D$, and $\Delta_J\geq f\circ\pi_J$ on $C$ with
  $J=\Supp D$. In particular $\cC\cup\cD$ is a partition of
  $\overline{A}$ which contains the faces of all its members, since
  $\cD$ is a closed complex (because $\cD$ is a complex and $\bigcup\cD=\partial
  B$ is closed). So $\cC\cup\cD$ is a closed complex.
\end{proof}

\begin{theorem}[Monohedral Decomposition]\label{th:mono-dec}
  Let $A\subseteq F_I(\Gamma^m)$ be a largely continuous precell mod $N$. Then there
  exists a complex $\cC$ of monohedral largely continuous precells mod $N$
  such that $A=\bigcup\cC$.
\end{theorem}

\begin{proof} 
We are going to show that given any closed complex $\cA$ of largely
continuous precells mod $N$ in $\Gamma^m$, there is a closed complex $\cC$ of
largely continuous {\em monohedral} precells mod $N$ such that
$\bigcup\cC=\bigcup\cA$ and $\cC$ refines $\cA$ (that is every $C\in\cC$ is
contained in some $A\in\cA$). The conclusion for $A$ will follow, by
applying this to the closed complex consisting of all the faces of
$A$. The proof goes by induction on the cardinality $n$ of $\cA$. If
$n=0$ then $\cC=\cA=\emptyset$ proves the result. Assume that $n\geq1$ and the
result is proved for smaller integers. Let $A$ be a maximal element of
$\cA$ with respect to specialisation, and $\cB=\cA\setminus\{A\}$. By maximality
of $A$, $\cB$ is again a closed complex. The induction hypothesis
gives a closed complex $\cD$ of largely continuous monohedral precells
mod $N$ such that $\bigcup\cD=\bigcup\cB$ and $\cD$ refines $\cB$. If $A$ is
closed then obviously $\cC=\cD\cup\{A\}$ proves the result for $\cA$.
Otherwise let $\cD_A=\{D\in\cD\tq D\subseteq\partial A\}$. The Monohedral Division
Theorem~\ref{th:mono-div} applied to $A$, $\cD_A$ and the constant
function $f=0$ gives a finite partition $\cC_A$ of $A$ in monohedral
largely continuous precells mod $N$ such that the family $\cC_A\cup\cD_A$
is a closed complex. The family $\cC=\cC_A\cup\cD$ is a partition of
$A\cup\bigcup\cB=\bigcup\cA$. Since $\cD$ is a closed complex and every precell in
$\cC_A$ has a unique facet which belongs to $\cD$, it follows that
$\cC$ is a complex. 
\end{proof}

We finish this section with another, much more elementary, division
result. Contrary to the above ones, it is drastically different from
what occurs in the real situation, where the polytopes are connected sets. 

\begin{proposition}\label{pr:split-mono}
  Let $A\subseteq F_I(\Gamma^m)$ be a non-closed monohedral largely continuous precell
  mod $N$. For every integer $n\geq1$ there exists for some $N'\in(\NN^*)^m$
  a partition $(A_i)_{1\leq i\leq n}$ of $A$ in largely continuous precells
  mod $N'$ such that $\partial A_i=\partial A$ for $1\leq i\leq n$.
\end{proposition}

\begin{proof}
The proof goes by induction on $m$. The result is trivially true for
$m=0$ since there is no non-closed precell in $\Gamma^0$. Assume that $m\geq1$
and the result is proved for smaller integers. Let $(\mu,\nu,\rho)$ be a
presentation of $A$. By induction hypothesis we can assume that
$m\in\Supp A$ hence $\mu<+\infty$. If $\nu=+\infty$, for $1\leq i\leq n$ let $A_i$ be the
set of $a\in F_I(\Gamma^m)$ such that $\widehat{a}\in\widehat{A}$,
$\mu(\widehat{a})\leq a_m\leq \nu(\widehat{a})$ and $a_m\equiv\rho+iN_m\;[nN_m]$. This
is obviously a partition of $A$ in largely continuous precells mod
$N'=(\widehat{N},nN_m)$ having the same boundaries as $A$. On the other
hand, if $\nu< +\infty$ then $\widehat{A}$ is not closed (otherwise $A$ would
be closed) hence the induction hypothesis gives for some $P'\in(\NN^*)^m$
a partition $(X_i)_{1\leq i\leq n}$ of $\widehat{A}$ in largely continuous
precells mod $P'$ such that $\partial X_i=\partial X$ for every $i$. Let
$A_i=(X_i\times\cZ)\cap A$ for every $i$. Then $(A_i)_{1\leq i\leq n}$ is easily
seen to give the conclusion, thanks to the description of the faces of
$A$ and $A_i$ given by Proposition~\ref{pr:pres-face}. 
\end{proof}

\section{Polytopes in $p$\--adic fields}
\label{se:p-adic}

Recall that $K$ is a $p$\--adically closed field, $v$ its
$p$\--valuation, $R$ its valuation ring and $\Gamma=v(K)$.
We still denote by $v$ the map $(v,\dots,v)$ from $K^m$ to $\Gamma^m$. 

We are going to define polytopes\footnote{We don't call them largely
  continuous precells because they are much more special than the
  usual $p$\--adic cells as defined in \cite{dene-1986}.} 
mod $N$ in $K^m$ by means of the inverse image by $v$ of largely
continuous precells mod $N$ in $\Gamma^m$. However, the $p$\--adic
triangulation theorem that we are aiming at requires a more versatile
definition. It involves semi-algebraic subgroups $Q_{1,M}$ of the
multiplicative group $K^\times=K\setminus\{0\}$, where $M$ is a positive integer. In
the special case where $K$ is a finite extension of $\QQ_p$, we have
\begin{displaymath}
  Q_{1,M}=\bigcup_{k\in\ZZ}\pi^k(1+\pi^MR).
\end{displaymath}
where $\pi$ is any generator of the maximal ideal of $R$. 
Since in this paper we will only use that $v(Q_{1,M})=\cZ$, we refer
the reader to \cite{cluc-leen-2012} for a general definition of
$Q_{N,M}$ for every integers $N,M\geq1$ in arbitrary $p$\--adically
closed fields. 

We let $D^MR=(\{0\}\cup Q_{1,M})\cap R$. Given an $m$\--tuple $N\in(\NN^*)^m$ we
call a set $S\subseteq K^m$ a {\df polytope mod $N$ in $D^MR^m$} if $v(S)$ is
a largely continuous precell mod $N$ in $\Gamma^m$ and $S=v^{-1}(v(S))\cap D^MR^m$. The {\df
faces} and {\df facets} $F_J(S)$ of a subset $S$ of $D^MR^m$ are
defined as the inverse images, by the restriction of $v$ to
$D^MR^m$, of the faces and facets of $v(S)$. The {\df support} of
$S$ (resp. of $x\in K^m$) is the support of $v(S)$ (resp. of $v(x$),
so that:
\begin{displaymath}
  \Supp(x)=\big\{i\in[\![1,m]\!]\tq x_i\neq0\big\} 
\end{displaymath}
\begin{displaymath}
  F_J(S)=\big\{x\in\overline{S}\tq \Supp x=J\big\}
\end{displaymath}
We say that $S$ is {\df monohedral} if $v(S)$ is so, that is if the
faces of $S$ are linearly ordered by specialisation, in which case we
call $S$ a {\df monotope mod $N$ in $D^MR^m$}. 

A family $\cC$ of polytopes mod $N$ in $D^MR^m$ is a {\df complex} if
it is finite and for every $S,T\in\cC$, $\overline{S}\cap\overline{T}$ is
the union of the common faces of $S$ and $T$. It is a {\df closed
complex} if moreover it contains all the faces of its members. Every
complex $\cS$ of polytopes mod $N$ is contained in a smallest closed
complex, namely the family of all the faces of the members of $\cS$.
We call it the {\df closure} of $\cS$ and denote it $\overline{\cS}$. 

In order to ease the notation, we write $vS$ for $v(S)$, and $v\cC$
for $\{vS\tq S\in\cC\}$. Clearly $\cC$ is a (closed) complex if
and only if $vC$ is.

\begin{proposition}\label{pr:face-preim}
  Let $S$ be a polytope mod $N$ in $D^MR^m$, and let $T=F_J(S)$ be any of
  its faces. Then $T$ is a polytope mod $N$ equal to $\pi_J(S)$. 
\end{proposition}

\begin{proof}
Due to the correspondence between the faces of $S$ and $vS$, this
follows directly from Proposition~\ref{pr:pres-face} and
Proposition~\ref{pr:face-egale-proj}(\ref{it:face-proj}).
\end{proof}

More generally, all the points of
Proposition~\ref{pr:face-egale-proj}, as well as
Proposition~\ref{pr:face-socle}, Corollary~\ref{co:socle-facet}, 
the Monohedral Decomposition (Theorem~\ref{th:mono-dec}) and
Proposition~\ref{pr:split-mono} immediately
transfer to polytopes mod $N$ in $D^MR^m$. Only the Monohedral Division
(Theorem~\ref{th:mono-div}) requires a bit more of preparation.

For the sake of generality we want the $p$\--adic analogon of the
Monohedral Division Theorem in $\Gamma^m$ to hold not only with a map $\varepsilon:\partial
S\subseteq K^m\to K^*$ definable in the language of rings ({\it i.e.}
semi-algebraic) but also with a map definable in various expansions
$(K,\cL)$ of the ring structure of $K$. The proof of
Theorem~\ref{th:div-p-adic} below shows that it suffices to make the
following assumptions on $(K,\cL)$: 
\begin{description}
  \item[(Ext)]
    For every definable function $f:X\subseteq K^m\to K^*$, if $f$ is continuous
    and $X$ is closed and bounded, then $v(f(x))$ takes a maximum
    value at some point $x\in X$.
  \item[(Pres)]
    The image by the valuation of every subset of $K^m$ definable in
    $(K,\cL)$, is $\cL_{Pres}$\--definable. 
\end{description}

\begin{remark}\label{re:P-min}
  If $K$ is a finite extension of $\QQ_p$ then
  condition~(Ext) holds for every continuous function by
  the Extreme Value Theorem. But this condition, when restricted to
  definable continuous functions, is preserved by elementary
  equivalence. Hence it will be satisfied whenever the complete theory
  of $(K,\cL)$ has a $p$\--adic model (that is a model whose
  underlying field is a finite extension of $\QQ_p$). On the other hand,
  if $(K,\cL)$ is $P$\--minimal (see \cite{hask-macp-1997}), Theorem~6
  in \cite{cluc-2003} proves that condition~(Pres) is
  satisfied. In particular Theorem~\ref{th:div-p-adic} applies for
  example to every subanalytic map $\varepsilon$, and more generally to every
  map $\varepsilon$ which is definable in a $P$\--minimal structure $(K,\cL)$
  which has a $p$\--adic model.
\end{remark}

For every $x\in K^m$ we let $w(x)=\min_{1\leq i\leq m}v(x_i)$. If $v(K)=\ZZ$
this is the valuative counterpart of the usual norm on $K^m$, which
measures the distance of $x$ to the origin (see also
Remark~\ref{re:delta-dist}). 

\begin{theorem}[Monotopic Division]\label{th:div-p-adic}
  Let $S$ be a polytope mod $N$ in $D^MR^m$, $\varepsilon:\partial S\to K^*$ a definable
  function (in some expansion $(K,\cL)$ of the ring structure of $K$
  satisfying previous conditions (Ext) and (Pres)). Let $\cT$ be a
  complex of monotopes mod $N$ in $D^MR^m$ such that $\bigcup\cT=\partial S$.
  Assume that the restriction of $v\circ\varepsilon$ to every proper face of $S$ is
  continuous. Then there exists a finite partition $\cU$ of $S$ in
  monotopes mod $N$ in $D^MR^m$ such that $\cU\cup\cT$ is a closed
  complex, $\cU$ contains for every $T\in\cT$ a unique monotope $U$ with
  facet $T$, and moreover for every $u\in U$
  \begin{displaymath}
    w\big(u - \pi_J(u)\big) \geq v\big(\varepsilon(\pi_J(u))\big)
  \end{displaymath}
  where $J=\Supp(T)$.
\end{theorem}

\begin{proof}
For every proper face $F_J(S)$ of $S$, and every $s\in F_J(S)$, the
function $t\mapsto v(\varepsilon(t))$ is continuous on $v^{-1}(\{v(s)\})\cap F_J(S)$,
which is a closed and bounded domain. Thus it attains a maximum value
$e(s)$ (see Remark~\ref{re:P-min}). So let 
\begin{displaymath}
  G_J=\big\{(s,t)\in F_J(S)\times K\tq v(t)=e(s)\big\}.
\end{displaymath}
This is a definable set hence $v(G_J)$ is $\cL_{Pres}$\--definable
(see Remark~\ref{re:P-min}). Moreover by construction $v(G_J)$ is the
graph of a function $g_J:vF_J(S)=F_J(vS)\to\cZ$, such that $v(\varepsilon(s))\leq g_J(v(s))$
for every $s\in S$. Let $g:\partial(vS)\to\cZ$ be the function whose restriction
to each $F_J(vS)$ is $g_J$.

The Monotopic Division (Theorem~\ref{th:mono-div}) applies to $vS$,
$g$ and $vT$. It gives a finite partition $\cC$ of $vS$ in monotopes
mod $N$ such that $\cC\cup v\cT$ is a complex, every non-closed $C\in \cC$
has a unique facet $D$ which belongs to $v\cT$ and $\Delta_J\geq g\circ\pi_J$ on $C$
where $J=\Supp D$. Let $\cU$ be the family of $v^{-1}(C)\cap D^MR^m$ for
$C\in\cC$. This is clearly a finite partition of $S$ in monotopes mod $N$
in $D^MR^m$. Every $U\in\cU$ has a unique facet $T\in\cT$, and $\Delta_J\geq
g\circ\pi_J$ on $vT$ where $J=\Supp vT=\Supp T$. That is, for every $u\in U$
we have 
\begin{equation}
  w\big(u - \pi_J(u)\big)= \min_{i\notin J}v(u_i) = \Delta_J(v(u))
  \geq g\circ\pi_J(v(u)) 
  \label{eq:div-p-adiv}
\end{equation}
By construction  $\pi_J(v(u))=v(\pi_J(u))$ and
$g(v(t))\geq v(\varepsilon(t))$ for every $t\in T$, hence
\begin{displaymath}
  g\circ\pi_J\big(v(u)\big)=g\big(v(\pi_J(u))\big)\geq v\big(\varepsilon(\pi_J(u))\big).
\end{displaymath}
Together with (\ref{eq:div-p-adiv}), this proves the last point. 
\end{proof}

Finally, let us mention for further works the following generalisation of
Proposition~\ref{pr:split-mono}.

\begin{proposition}\label{pr:split-p-adic}
  Let $A\subseteq D^MR^m$ be a relatively open\footnote{A subset $A$ of a
    toplogical set is called {\df relatively open} if it is open in
    its closure, that is $\overline{A}\setminus A$ is closed.} set. Assume
  that $A$ is the union of a complex $\cA$ of monotopes mod $N$ in
  $D^MR^m$. Then for every integer $n\geq1$ there exists a finite
  partition of $A$ in semi-algebraic sets $A_1,\dots,A_n$ such that $\partial
  A_k=\partial A$ for every $k$. 
\end{proposition}

\begin{proof}
Thanks to the correspondence between the faces of the monotopes mod
$N$ in $D^MR^m$ and their faces, it suffices to prove the result for
a relatively open set $A\subseteq\Gamma^m$ which is the union of a complex of
monotopes mod $N$ in $\Gamma^m$. 

Let $\cC=\overline{\cA}\setminus\cA$ and $C=\bigcup\cC=\overline{A}\setminus A$. By
assumption $A$ is relatively open hence $C$ is closed, so $\cC$ is a
closed complex. Let $U_1,\dots,U_r$ be the list of minimal elements of
$\cA$. Every $S\in\overline{\cA}$ such that $U_i\leq S$ for some $i$
belongs to $\cA$ (otherwise $S\in\overline{\cA}\setminus\cA=\cC$ which is
closed, hence $U_i\in\cC$, a contradiction since $U_i\in\cA$). Note
further that every $T\in\overline{\cA}\setminus\cA$ is a proper face of some $U_i$
(because $T$ is a face of some $S\in\cA$ and $U_i\leq S$ for some $i$, hence
$T< U_i$ or $U_i\leq T$ because $S$ is a monotope, and the second case is
excluded because $T\notin\cA$). In particular $\partial A=\overline{A}\setminus A$ is the
union of the sets $T\in\overline{\cA}$ such $T<U_i$ for some $i$, that
is $\partial A=\bigcup_{i\leq r}\partial U_i$. 

For each $i\leq r$ let $\cB_i$ be the family of $S\in\cA$ such that $S\geq
U_i$, and $B_i=\bigcup\cB_i$. The families $\cB_i$ are pairwise disjoint, and
so are the sets $B_i$ since $\cA$ is a complex. By the same argument
as above (replacing $\cA$ by $\cB_i$) $\overline{B}_i\setminus B_i
=\bigcup(\overline{\cB}_i\setminus\cB_i) =\partial U_i$, hence $B_i$ is relatively open
and $\partial B_i=\partial U_i$. It suffices to prove the result separately for each
$B_i$. Indeed, assume that for each $i\leq r$ we have found a partition
$(B_{i,j})_{1\leq j\leq n}$ of $B_i$ in definable sets such that $\partial
B_{i,j}=\partial B_i$. Then let $A_j=\bigcup_{i\leq r}B_{i,j}$ for each $j$. By
construction these sets form a partition of $A$ and 
\begin{displaymath}
  \overline{A}_j\setminus A_j 
  = \overline{A}_j\setminus A 
  = \bigcup_{i\leq r}\overline{B}_{i,j}\setminus A 
  = \bigcup_{i\leq r}\overline{B}_i\setminus A
  = \bigcup_{i\leq r}\partial B_i 
  = \partial A. 
\end{displaymath}

Thus replacing $A$ and $\cA$ by $B_i$ and $\cB_i$ if necessary, we can
assume that $\cA$ has a unique smallest element $U_0$. If $U_0$ is
closed, then $\partial A=\partial U_0=\emptyset$ (by minimality of $U_0$), and it suffices
to take $A_1=A$ and $A_k=\emptyset$ for $k\geq2$. So from now on we assume that
$U_0$ is not closed. Proposition~\ref{pr:split-p-adic} then
applies to $U_0$ and gives for some $N'$ a partition
$A_{1}(U_0),\dots,A_{n}(U_0)$ of $U_0$ in largely continuous monotopic
cells mod $N'$ such that $\partial A_i(U_0)=\partial U_0$ for every $i$. In
particular each $A_i(U_0)$ is a basic Presburger set. Let $H=\Supp
U_0$, and for every $S\in\cS$ and $i\in[\![1,n]\!]$ let
$A_i(S)=\pi_H^{-1}(A_i(U_0))\cap S$. Note that this is a basic Presburger
set. Indeed, $S$ itself is a basic Presburger set, and
$\pi_H^{-1}(A_i(U_0))\cap F_I(\Gamma^m)$ with $I=\Supp S$ is a basic Presburger
set because  $A_i(U_0)$ is so (replace every condition $f(x)\geq0$
defining $A_i(U_0)$ by $f\circ\pi_H(x)\geq0$). Hence their intersection
$A_i(S)$ is a basic Presburger set too. For every $i\leq n$ let
$A_i=\bigcup\{A_i(S)\tq S\in\cA\}$. This defines a partition of $A$. In order
to conclude it only remains to show that $\overline{A_i}=A_i\cup\partial U_0$
for each $i$, so that $\partial A_i=\partial U_0=\partial A$. Since
$\overline{A_i}=\bigcup\{\overline{A_i(S)}\tq S\in\cA\}$, it suffices to check
that for every $S\in\cA$
\begin{equation}
  \overline{A_i(S)}=\bigcup\{A_i(T)\tq T\in\cA,\ U_0\leq T\leq S\}\cup\partial U_0. 
  \label{eq:split-p-adic}
\end{equation}
Let $I=\Supp A_i(S)=\Supp S$, and $J\subseteq I$ be the support of any face of
$A_i(S)$. Note that $F_J(A_i(S))\neq\emptyset$ implies that $F_J(S)\neq\emptyset$, thus $J$
is the support of a face $T=F_J(S)$ of $S$. This face $T$ belongs to
$\overline{\cA}$, hence to $\cS$ if $U_0\leq T$. We are claiming that
$F_J(A_i(S))=A_i(T)$ in that case, and that $F_J(T)=T=F_J(U_0)$ if
$T<U_0$. This will finish the proof since $\overline{A_i(S)}$ is the
union of its faces, and (\ref{eq:split-p-adic}) then follows
immediately.

Assume first that $U_0\leq T$. Since $A_i(S)$ and $S$ are basic
Presburger set, we know by
Proposition~\ref{pr:face-egale-proj}(\ref{it:face-proj}) that
$F_J(A_i(S))=\pi_J(A_i(S))$ and $F_J(S)=\pi_J(S)$, that is $T=\pi_J(S)$.
Since $U_0\leq T$ we have $H\subseteq J$ hence
$\pi_J(\pi_H^{-1}(A_i(U_0)))=\pi_J^{-1}(A_i(U_0))$. It follows that
\begin{displaymath}
  \pi_J\big(\pi_H^{-1}(A_i(U_0))\cap S\big)
  \subseteq \pi_J\big(\pi_H^{-1}(A_i(U_0))\big)\cap\pi_J(S)
  = \pi_J^{-1}(A_i(U_0)) \cap T
\end{displaymath}
that is $\pi_J\big(A_i(S)\big) \subseteq A_i(T)$.

Conversely, for every $y\in A_i(T)$ we have on one hand $y\in T=\pi_J(S)$ so
there is $x\in S$ such that $\pi_J(x)=y$, and on the other hand
$y\in\pi_H^{-1}(A_i(U_0))$ so $\pi_H(x)=\pi_H(\pi_J(x))=\pi_H(y)\in A_i(U_0)$. Thus
$x\in\pi_H^{-1}(A_i(U_0))\cap S=A_i(S)$, and since $y=\pi_J(x)$ this proves
that $A_i(T)\subseteq\pi_J(A_i(S))$. This proves our claim in this case.

Now assume that $T<U_0$. Then $J\subset H$ hence
$\pi_J(A_i(S))=\pi_J(\pi_H(A_i(S))$. We already know that
$\pi_H(A_i(S))=A_i(U_0)$ be the previous case, and that $\partial A_i(U_0)=\partial
U_0$ by construction. In particular $F_J(A_i(U_0))=F_J(U_0)$. But
$F_J(U_0)=T$ since $\cA$ is a complex and $T<U_0$. Altogether, using
Proposition~\ref{pr:face-egale-proj}(\ref{it:face-proj}) for $A_i(S)$
and $A_i(U_0)$ we get
\begin{displaymath}
  F_J(A_i(S))=\pi_J(A_i(S))
  =\pi_J(A_i(U_0))=F_J(A_i(U_0))=T.
\end{displaymath}
\end{proof}

% \begin{remark}\label{re:rem-referee}
%   As explained in the introduction, we called our sets polytopes
%   because of the numerous analogies that they share with real
%   polytopes, and specially because of the convenient notion of faces
%   which is attached to them, with all its topological properties.
%   However the notion of convexity, which is important in the real
%   case, disappears\footnote{Precells mod $N=(1,\dots,1)$, are convex sets
%   in $\ZZ^m$ since the are defined by linear inequalities.} in the
%   $p$-adic context. For example one may consider for every finite part
%   $A$ of $K^m$ the `convex hull' of $A$ as defined by
%   \begin{displaymath}
%     H(A)=\bigg\{\sum_{a\in A}\lambda_aa\mbox{ such that each }\lambda_a\in R\setminus\{0\}
%     \mbox{ and }\sum_{a\in A}\lambda_a=1\bigg\}
%   \end{displaymath}
%   These sets $H(A)$ may be formally closer to real polytopes. In
%   particular it would be natural to call 'faces' of $A$ the sets
%   $H(B)$ for $B\subseteq A$, and to expect the frontier of $H(A)$ to be
%   covered by these faces. But this is rarely the case: for example
%   if $A=\{(0,0),(1,0),(0,1),(1,1)\}$ then $H(A)=R^2$ has an empty
%   frontier, and is not disjoint from its 'faces'.
% \end{remark}

\bibliographystyle{alpha}
% \bibliography{biblio}

\end{document}